\documentclass[final,12pt]{colt2021}
\usepackage[utf8]{inputenc}
\usepackage[T1]{fontenc}
\usepackage{amsmath}
\usepackage{amsfonts}
\usepackage{amssymb}
\usepackage{enumerate}
\usepackage{bbm}
\usepackage{graphicx}
\usepackage{float}
\usepackage{caption}
\usepackage{microtype}

\usepackage[capitalise]{cleveref}
\usepackage{autonum} 

\newtheorem{assumption}{Assumption}
\crefname{remark}{Remark}{Remarks}
\Crefname{remark}{Remark}{Remarks}

\makeatletter
\def\set@curr@file#1{\def\@curr@file{#1}} 
\makeatother
\usepackage[load-configurations=version-1]{siunitx} 

\title{A Dimension-free Computational Upper-bound for Smooth Optimal Transport Estimation}

\coltauthor{
\Name{Adrien Vacher} \Email{adrien.vacher@u-pem.fr}\\
\addr
INRIA Paris, 2 rue Simone Iff, 75012, Paris, France\\
LIGM, Université Gustave Eiffel, CNRS
\AND
\Name{Boris Muzellec} \Email{boris.muzellec@inria.fr}\\
\addr
INRIA Paris, 2 rue Simone Iff, 75012, Paris, France\\
Département d’Informatique, École Normale Supérieure, Paris, France\\
PSL Research University, 2 rue Simone Iff, 75012, Paris, France
\AND
 \Name{Alessandro Rudi} \Email{alessandro.rudi@inria.fr}\\
 \addr
 INRIA Paris, 2 rue Simone Iff, 75012, Paris, France\\
Département d’Informatique, École Normale Supérieure, Paris, France\\
PSL Research University, 2 rue Simone Iff, 75012, Paris, France
 \AND
 \Name{Francis Bach} \Email{francis.bach@inria.fr}\\
 \addr
 INRIA Paris, 2 rue Simone Iff, 75012, Paris, France\\
Département d’Informatique, École Normale Supérieure, Paris, France\\
PSL Research University, 2 rue Simone Iff, 75012, Paris, France
\AND
 \Name{François-Xavier Vialard} \Email{francois-xavier.vialard@u-pem.fr}\\
 \addr
LIGM, Université Gustave Eiffel, CNRS
}

\usepackage{notations}

\date{\today}

\begin{document}

\maketitle

\begin{abstract}
    It is well-known that plug-in statistical estimation of optimal transport suffers from the curse of dimensionality. Despite recent efforts to improve the rate of estimation with the smoothness of the problem, the computational complexity of these recently proposed methods still degrade exponentially with the dimension. In this paper, thanks to an infinite-dimensional sum-of-squares representation, we derive a statistical estimator of smooth optimal transport which achieves a precision $\varepsilon$ from $\tilde{O}(\varepsilon^{-2})$ independent and identically distributed samples from the distributions, for a computational cost of $\tilde{O}(\varepsilon^{-4})$ when the smoothness increases, hence yielding dimension-free statistical \emph{and} computational rates, with potentially exponentially dimension-dependent constants.   
\end{abstract}

\section{Introduction}\label{sec:intro}

\BM{I've added the autonum package, which allows to only number referenced equations}
The comparison between probability distributions is a fundamental task and has been extensively used in machine learning. 
For this purpose, optimal transport (OT) has recently gained traction in different subfields of machine learning (ML), such as natural language processing (NLP) \citep{xu2018distilled,chen2018adversarial}, generative modeling \citep{arjovsky2017wasserstein,tolstikhin2018wasserstein,salimans2018improving}, multi-label classification \citep{frogner2015learning}, domain adaptation \citep{redko2019domain}, clustering \citep{ho2017multilevel}, and has had an impact in other areas such as imaging sciences \citep{Feydy2017,bonneel2011displacement}. 
Indeed, OT is a tool to compare data distributions which has arguably many more geometric properties than other available divergences \citep{peyre2019computational}. 

In practice, the optimal transport cost is often computed for the squared distance (leading to the Wasserstein-2 distance) on sampled distributions with $n$ observations, and it is well-known that optimal transport suffers from the curse of dimensionality \citep{fournier2015rate}: the plug-in strategy, which simply consists in computing the Wasserstein distance between the sampled distributions, yields an estimation of the Wasserstein squared distance between a density and its sampled version in $O(1/n^{1/d})$, which degrades rapidly in high dimensions; this can only be improved to $O(1/n^{2/d})$ in the case of two different distributions~\citep{chizat2020faster}.

However, high dimension is the usual setting in machine learning, such as in NLP~\citep{grave2019unsupervised}, and even if the intrinsic dimensionality of data can be leveraged~\citep{weed2019sharp,weed2019spiked}\BM{fixed references here}, poor theoretical rates of convergence are a recurrent feature of OT. 
\BM{Changed this paragraph to add references to \citep{liang2018HowWell} and \citep{liang2019EstimatingCertain}}
\citet{liang2018HowWell,liang2019EstimatingCertain} recently showed that when the measures admit smooth densities, the Wasserstein-1 distance (as part of a more general class of integral probability metrics (IPM)) those minimax sample complexity rates could be improved to almost $O(1/\sqrt{n})$ when the smoothness increases.
\citet{weed2019estimation} then showed equivalent rates in the case smooth densities with geometric assumptions on their supports for the Wasserstein-$p$ distances with $p > 1$, which are not IPMs, and proposed a corresponding estimator based on a dedicated non-polynomial-time algorithm.
Matching rates were then proved for the transportation maps themselves~\citep{hutter2019minimax} in the Wasserstein-2 setting, under smoothness assumptions on those maps. This line of work is deeply related to the regularity theory of optimal transport, that guarantees the smoothness of the optimal map in Euclidean spaces under similar assumptions on the source and target distributions, and their supports~\citep{caffarelli1992regularity,philippis2013mongeampre}. Yet, to this day no practically tractable algorithm (e.g., with polynomial time) matching the bounds of \citet{weed2019estimation} and \citet{hutter2019minimax} is known.

\paragraph{State of the art.}
An approach that has first been  advocated in the machine learning community as a way to efficiently approximate empirical OT and to make it differentiable consists in adding an entropic regularization term to the OT problem~\citep{cuturi2013sinkhorn}. Rates on the sample complexity of entropic optimal transport have then been studied by \citet{genevay2019sample} and \citet{mena2019statistical}, and were proven to be of the order $O(\frac{1}{\varepsilon
^{\lfloor d/2\rfloor}\sqrt{n}})$, for small values of $\varepsilon$. Although the dependency in the number of samples is in $1/\sqrt{n}$, the constant degrades exponentially with respect to the dimension.
Entropic OT and Sinkhorn divergences~\citep{genevay2018learning} were then leveraged as a tool to study the sample complexity of the (unregularized) OT problem itself:
the most advanced results in this direction were derived by \citet{chizat2020faster}, who show that with few assumptions on the regularity of the Kantorovich potentials, 
the squared Wasserstein distance can be estimated using $O(\varepsilon^{-d/2 + 2})$ samples and $\tilde{O}(\varepsilon^{-(d' + 5.5)})$ operations (with $d' = 2 \lfloor \frac d 2 \rfloor$) with high probability.
\par
\BM{I've added references to \citep{liang2018HowWell} and \citep{liang2019EstimatingCertain} in this paragraph.}
Our work can be related to a current research direction which consists in developing estimators of the Wasserstein distance for classes of smooth distributions, with smoothness parameter $m$, that have better performances than in the general case. Related to this trend, \citet{liang2018HowWell,liang2019EstimatingCertain} showed minimax rates for a class of integral probability metrics (IPM) that includes the Wasserstein-1 distance, as a function of the smoothness of the distributions. However, (i) for $p > 1$, the Wasserstein-$p$ distance $W_p$ is not an IPM and (ii) no estimators with matching rates are proposed in those two works. So far, two main contributions leveraging smoothness that are applicable to the $W_2$ distance can be found. \citet{hutter2019minimax} derive minimax rates for the estimation of the OT maps and propose an estimator which necessitates, for an $L_2$ error on the maps of order $\varepsilon$, $O(\varepsilon^{-\frac{2m - 1 + d/2}{2m}})$ samples. While statistically almost optimal, this estimator is not computationally feasible as it requires to project the potentials on a space of smooth, strongly-convex functions. Instead, \citet{weed2019estimation}  derive estimators for the densities requiring $O(\varepsilon^{-\frac{d + 2m}{1+m}})$ samples and, under the assumption that an efficient resampler is available, derive an estimator of the OT distance that can be calculated in $\tilde{O}(\varepsilon^{-(2d + d/2)})$ time.
\par While the contributions above do succeed in taking advantage of the smoothness from a statistical point of view, they do not manage to take advantage of the smoothness from a computational point of view. Actually, statistical-computational gaps are known to exist for some instances of high-dimensional OT, such as the spiked transport model of \citet{weed2019spiked}.

\paragraph{Contributions.} 
In this paper, we bridge the statistical-computational gap of smooth OT estimation and we provide a positive answer to the question whether smoothness of the optimal potentials can be computationally beneficial to an efficient statistical estimator. 
More precisely, we propose an algorithm which, for a given accuracy $\varepsilon$, needs $O(\varepsilon^{-2})$ samples and has a computational complexity of $\tilde{O}(\varepsilon^{-\max(4,\frac{7d}{m-d})})$. Note that the computational complexity improves with the regularity of the distributions and, when $m \geq 3d$, it is $\tilde{O}(\varepsilon^{-4})$, i.e., independent of the dimension $d$ in the exponent (but not in the constants). We thus show that smoothness can be leveraged in the computational estimation of optimal transport.

Moreover, we consider different scenarios beyond i.i.d.~sampling, such as the case where we are able to compute  exact integrals or where we can evaluate the densities in given points, by representing the problem in terms of kernel mean embeddings~\citep{muandet2016kernel}. This allows to make a unified analysis for all the cases. The total error is then the sum of the error induced by approximating via the kernel mean embedding plus the error induced by subsampling the constraints. 
Interestingly, 
in the other scenarios the computational cost to achieve an error $\eps$ can be smaller than $\eps^{-4}$, as reported below.
This is particularly interesting in the case we can evaluate the densities in given points and avoids using expensive Monte-Carlo sampling techniques to obtain i.i.d.~samples (see \cref{sec:mean_estimators} for more details). Our results are summarized below. 

\begin{theorem}\label{thm:complexities}
Let $\eps > 0$. Let $\mu, \nu$ satisfy \cref{assum:measures} for some $m > d$. Let $\hOT$ be the proposed estimator defined in \cref{eq:hat_OT} and computed as in \cref{sec:algo} with the same parameters as in Corollary~\ref{cor:costs}. The cost to achieve $|\hOT - \OT(\mu, \nu)| \leq \eps$ for the three scenarios is:
\begin{enumerate}
\item (Exact integral) Time: $\tilde{O}(\eps^{-\frac{7d}{m-d}})$. Space: $\tilde{O}(\eps^{-\frac{4d}{m-d}})$.
\item (Evaluation) Time: $\tilde{O}(\eps^{-\frac{7d}{m-d}})$. Space: $\tilde{O}(\eps^{-\frac{4d}{m-d}})$. $\#$evaluations of $\mu,\nu$: $\tilde{O}(\eps^{-\frac{d}{m+1}})$.
\item (Sampling) Time: $\tilde{O}(\eps^{-\max(4,\frac{7d}{m-d})})$. Space: $\tilde{O}(\eps^{-\frac{4d}{m-d}})$.  $\#$samples of $\mu,\nu$: $\tilde{O}(\eps^{-2})$.
\end{enumerate}
\end{theorem}
The second key contribution of this paper is to provide a new representation theorem for solutions of smooth optimal transport. The inequality constraint in the dual OT problem can be replaced with an equality constraint involving a finite sum-of-squares in a Sobolev space. In comparison with \cite{rudi2020global}, it is a non-trivial extension of their representation result to the case of a \emph{continuous} set of global minimizers instead of a \emph{finite} set.

\section{Sketch of the result and derivation of the algorithm}\label{sec:sketch}
In this paper, we consider the optimal transport problem for the quadratic cost on bounded subsets $X, Y$ of the Euclidean space $\R^d$. The set of probability measures on $X$ is denoted by $\mathcal{P}(X)$.
The optimal transport problem with quadratic cost $c(x,y) = \frac{1}{2}\|x-y\|^2$ can be stated in its dual formulation as
 \begin{align}
 \begin{split}
    \label{EqDualOT}
    \OT(\mu, \nu)  =   \sup_{u,v \in C(\R^d)} ~&~~ \int u(x)d\mu(x)  + \int v(y) d\nu(y) \\
    \textrm{subject to}  ~&~~  c(x,y) \geq u(x) + v(y), ~~ \forall (x,y) \in X \times Y,
\end{split}
\end{align}
As a standard result in optimal transport theory, the supremum is attained and the functions $u_\star,v_\star$ are referred to as the Kantorovich potentials \citep[see][]{santambrogio2015optimal}. 

The proposed approach to approximate $\OT(\mu, \nu)$ is the result of two main ingredients: (1) a suitable way to represent smooth functions and to approximate their integral in $\mu, \nu$, (2) a way to enforce efficiently the dense set of constraints on $u,v$.

\paragraph{Preliminary step: Representing smooth functions and integrals.}
We represent smooth functions via a {\em reproducing Kernel Hilbert space} (RKHS) \citep{aronszajn1950theory,steinwart2008support}, for which functions can be represented as linear forms.
In \Cref{sec:sos-rep} we show that under smoothness assumptions on $\mu$ and $\nu$ (\cref{assum:measures}) we have $u \in \hhx$ and $v \in \hhy$ where $\hhx$ and $\hhy$ are two suitable RKHSs on $X$ and $Y$, associated with two bounded continuous feature maps $\phi_X: X \to \hhx$ and $\phi_Y: Y \to \hhy$. Note that RKHSs offer several advantages. First, leveraging the reproducing property, we can represent the integrals in the functional of \cref{EqDualOT} as inner products in terms of the kernel mean embeddings $w_\mu \in \hhx$ and  $w_\nu \in \hhy$ where $w_\mu = \int_X \phi_X(x) d\mu(x)$ and $w_\nu = \int_Y \phi_Y(y) d\nu(y)$. Indeed, by the reproducing property, for all $u \in \hhx$, we have:
$$\int_X u(x) d\mu(x) = \int_X \scal{u}{\phi_X(x)}_\hhx d\mu(x) = \big<u, \int_X \phi_X(x) d\mu(x) \big>_\hhx = \scal{u}{w_\mu}_\hhx,$$
and the same reasoning holds for the integral on $\nu$, i.e., $\int_Y v(y) d\nu(y) = \scal{v}{w_\nu}_\hhy$, for all $v \in \hhy$. This construction is known as {\em kernel mean embedding} \citep{muandet2016kernel}. Moreover, RKHSs allow the so-called {\em kernel trick} \citep{steinwart2008support}, i.e., to express the resulting algorithm in terms of {\em kernel functions} that in our case correspond to $k_X(x,x') = \scal{\phi_X(x)}{\phi_X(x')}_\hhx$ and $k_Y(y,y') = \scal{\phi_Y(y)}{\phi_Y(y')}_\hhy$, that are known explicitly and are easily computable in $O(d)$.

\paragraph{The main step: Dealing with a dense set of inequalities.}
Even assuming that we are able to compute  integrals in closed form and restricting to $m$-times differentiable $u, v$, the main challenge is to deal with the dense set of inequalities $c(x,y) \geq u(x) + v(y)$ that $u,v$ must satisfy, for any $(x,y) \in X \times Y$. Indeed, an intuitive approach would be to subsample the set, i.e., to take $\ell$ points $(\tilde{x}_1,\tilde{y}_1),\dots,(\tilde{x}_\ell, \tilde{y}_\ell)$ in $X \times Y$ and consider only the constraints $c(\tilde{x}_j,\tilde{y}_j)  \geq u(\tilde{x}_j) + v(\tilde{y}_j)$ for $j = 1,\dots, \ell$. This approach, however, is only able to leverage the Lipschitzianity of $u, v$ \citep{aubin2020hard} and leads to an error in the order of $\ell^{-1/d}$ that does not allow to break the dependence in $d$ in the exponent, and yields rates that are equivalent to the plugin estimator.

In this paper, we leverage a more refined technique to approximate the dense set of inequalities, introduced by \cite{rudi2020global} for the problem of non-convex optimization, and that allows to break the curse of dimensionality for smooth problems. The idea behind this technique is the consideration that, while a dense set of inequalities is poorly approximated by subsampling, the situation is different in the case of a dense set of \emph{equality} constraints, for which an optimal rate of $O(\ell^{-m/d})$ is achievable for $m$-times differentiable constraints \citep{wendland2005}. The construction works in two steps: first, substitute the inequality constraints with equality constraints that are equivalent, and then subsample. In the next two paragraphs we explain how to apply this approach to the problem of OT. 

\paragraph{Removing the inequalities: positive definite operator characterization.}
To apply the construction recalled above to our scenario, we first consider the following problem. 
Let $\hhxy$ be a Hilbert space on $X \times Y$ and $\phi:X \times Y \to \hhxy$. Denote by $k_{XY}$ the kernel $k_{XY}((x,y),(x',y')) = \scal{\phi(x,y)}{\phi(x',y')}_\hhxy$ for any $(x,y), (x',y') \in X \times Y$ and by $\pdm{\hhxy}$ the space of positive operators on $\hhxy$. We define
\begin{align}\label{eq:W-A}
\begin{split}
    \max_{\substack{u \in \hhx, v \in \hhy,\\ A \in \pdm{\hhxy}}} ~~~ & \scal{u}{w_\mu}_\hhx + \scal{v}{w_\nu}_\hhy \\
    \textrm{subject to} ~~~ &\forall (x,y) \in X \times Y, ~~~ c(x,y) - u(x) - v(y) = \scal{\phi(x,y)}{A \phi(x,y)}_\hhxy,
\end{split}
\end{align}
where the inequality in \eqref{EqDualOT} is substituted with an equality w.r.t.\ a positive definite operator $A$ on $\hhxy$. Note that Problem \eqref{EqDualOT} is a relaxation of Problem \eqref{eq:W-A}: indeed, if for a given pair $u \in \hhx, v \in \hhy$ there exists a positive definite $A$ satisfying the equality above, then 
$$ c(x,y) - u(x) - v(y) = \scal{\phi(x,y)}{A \phi(x,y)}_\hhxy \geq 0, \qquad \forall (x,y) \in X \times Y,$$
so the couple $(u, v)$ is admissible for \eqref{EqDualOT}. However, even for an admissible couple in \eqref{EqDualOT} satisfying $u\in \hhx, v \in \hhy$, a positive operator $A$ may not exist. Indeed, note that the technique of representing a positivity constraint in terms of a positive matrix has a long history in the community of polynomial optimization \citep{lasserre2001,parrilo2003semidefinite,LasserreBook}, which shows that in general the resulting problem is not equivalent to the original one, for any chosen degree of polynomial approximation. This fact leads to the so-called {\em sum of squares hierarchies}, also used for optimal transport \citep{henrion2020graph}. Instead, using kernels, \cite{rudi2020global} showed that there exists a positive operator with finite rank that matches the constraints and makes the two problems equivalent, when the constraint is attained on a finite set of points. However, such existence results cannot be used for the problem in \eqref{eq:W-A}, since in the case of optimal transport the set of zeros corresponds to the graph of the optimal transport map and is a smooth manifold, when $\mu, \nu$ are smooth \citep{philippis2013mongeampre}.

A crucial point of our contribution is then to prove that, with a quadratic cost $c(x, y) = \frac{1}{2} \|x-y\|^2$ and under the same assumptions on the densities and their supports, or under smoothness assumptions on the Kantorovich potentials, there exists a positive operator on a suitable Hilbert space that represents the function $c(x,y) - u(x) - v(y)$ for a pair $u,v$ maximizing~\eqref{EqDualOT}, making the two problems equivalent. The result is reported in \Cref{thm:SoS-OT}. The proof is derived using the Fenchel dual characterization of $u_\star, v_\star$ and gives a sharp control of the rank of $A$.

\paragraph{Subsampling the constraints and approximating the integrals.}
We restrict the constraint of \eqref{eq:W-A} to $(\tilde{x}_1,\tilde{y}_1), \dots, (\tilde{x}_\ell,\tilde{y}_\ell) \subset X \times Y$ for $\ell \in \NN$. However, we need to add a penalization for $u, v$ and $A$ to avoid overfitting, since the error induced by subsampling the constraints is proportional to the trace of $A$ \citep{rudi2020global} and, in our case, also to the norms of $u, v$, as derived in \Cref{thm:sample_bound} in \Cref{sec:sample_bounds}. Finally, in two of the three scenarios of interest for the paper, i.e., (i) when we can only evaluate $\mu, \nu$ pointwise, or (ii) when we have only i.i.d.~samples from $\mu, \nu$,  we do not have access to the kernel mean embeddings $w_\mu \in \hhx, w_\nu \in \hhy$. Therefore, we need to use some estimators $\hat{w}_\mu \in \hhx, \hat{w}_\nu \in \hhy$ that are derived in \Cref{sec:mean_estimators}. The resulting problem is the following, for some regularization parameters $\la_1, \la_2 > 0$:
\begin{align}\label{eq:w_hat}
\begin{split}
    \max_{\substack{u \in \hhx, v\in \hhy, \\ A \in \pdm{\hhxy}}} ~~~ & \scal{u}{\hat{w}_\mu}_\hhx + \scal{v}{\hat{w}_\nu}_\hhy - \la_1 \tr(A) - \la_2(\|u\|^2_\hhx + \|v\|^2_\hhy) \\
    ~~~ \textrm{subject to} ~~~ &\forall j \in [\ell], ~~ c(\tilde{x}_j, \tilde{y}_j) - u(\tilde{x}_j) - v(\tilde{y}_j) = \scal{\phi(\tilde{x}_j,\tilde{y}_j)}{A \phi(\tilde{x}_j,\tilde{y}_j)}_\hhxy.
\end{split}
\end{align}
Let $\hat{u}, \hat{v}$ be the maximizers of the problem above (unique since the problem is strongly concave in $u, v$). The estimator for $\OT$ we consider corresponds to
\begin{equation}\label{eq:ot_hat_primal}
    \hOT = \scal{\hat{u}}{\hat{w}_\mu}_\hhx + \scal{\hat{v}}{\hat{w}_\nu}_\hhy.
\end{equation}

\paragraph{Finite-dimensional characterization.}
 In \Cref{sec:algo}, following \cite{marteau2020non}, we derive the dual problem of \cref{eq:w_hat}. Define ${\bf Q} \in \RR^{\ell \times \ell}$ as ${\bf Q}_{i,j} = k_X(\tilde{x}_i, \tilde{x}_j) + k_Y(\tilde{y}_i, \tilde{y}_j)$ and $z_j = \hat{w}_\mu(\tilde{x}_j) + \hat{w}_\nu(\tilde{y}_j) - 2 \la_2 c(\tilde{x}_j,\tilde{y}_j)$ for $i,j \in [\ell]$ and $q^2 =  \|\hat{w}_\mu\|^2_\hhx + \|\hat{w}_\nu\|^2_\hhy$, and let $\bf I_\ell \in \RR^{\ell \times \ell}$ be the identity matrix. Let ${\bf K} \in \RR^{\ell \times \ell}$ be defined as ${\bf K}_{i,j} = k_{XY}((\tilde{x}_i,\tilde{y}_i), (\tilde{x}_j, \tilde{y}_j))$ and define $\Phi_i \in \RR^\ell$ as the $i$-th column of $\bf R$, the upper triangular matrix corresponding to the Cholesky decomposition of $\bK$ (i.e., $\bf R$ satisfies ${\bf K} ={\bf R}^\top {\bf R}$). The dual problem writes:
 \begin{align}\label{eq:w_hat_dual}
\begin{split}
     \underset{\gamma \in \RR^\ell}{\min}& ~\tfrac{1}{4\lambda_2} \gamma^\top {\bf Q} \gamma - \tfrac{1}{2\lambda_2} \textstyle \sum_{j =1}^\ell \gamma_j z_j + \tfrac{q^2}{4\la_2} ~~~ \mbox{ such that } ~~~ \textstyle \sum_{j =1}^\ell \gamma_j \Phi_j\Phi_j^\top + \lambda_1 \eye_\ell \succeq 0.
 \end{split}
\end{align}
In the same section in \Cref{sec:algo}, we derive an explicit characterization of $\hat{u}, \hat{w}, \hat{A}$ in terms of $\hat{\gamma}$, the solution of the problem above and we characterize $\hOT$ as follows:
\begin{equation}\label{eq:hat_OT}
\hOT = \tfrac{q^2}{2\lambda_2} - \tfrac{1}{2\la_2} \textstyle \sum_{j = 1}^\ell \hat{\gamma}_j (\hat{w}_\mu(\tilde{x}_j) + \hat{w}_\nu(\tilde{y}_j)).
\end{equation}
As it is possible to observe from the problem above and the characterization of $\hOT$, the only quantities necessary to compute $\hat{\gamma}$ and $\hOT$ are the kernels $k_X, k_Y, k_{XY}$ and the evaluation of the functions $\hat{w}_\mu \in \hhx, \hat{w}_\nu \in \hhy$ at the points $\tilde{x}_j$ and $\tilde{y}_j$ respectively for $j \in [\ell]$. In \Cref{sec:algo}, we consider a Newton method with self-concordant barrier to solve the problem above \citep{nesterov1994interior}. To illustrate that this algorithm can indeed be implemented in practice, we run simulations on toy data in \Cref{sec:numerical}. The total cost of the procedure to achieve error $\eps$ for the computation of $\hOT$ is the following (see \cref{thm:estimator_precision}, \cpageref{thm:estimator_precision} in the appendix):
\begin{equation}\label{eq:cost}
O\big(C +  E\ell + \ell^{3.5} \log \tfrac{\ell}{\eps}\big) ~~\textrm{time}, \qquad O(\ell^2) ~~ \textrm{memory},
\end{equation}
where $C$ is the cost for computing $q^2$ and $E$ is the cost to compute one $z_j$. Depending on the operations that we are able to perform on $\mu, \nu$ and $k_X, k_Y$, we have three scenarios. In \Cref{sec:mean_estimators} we specify how to compute the vectors $\hat{w}_\mu, \hat{w}_\nu$, in Corollary \ref{cor:costs} we report only the conditions of applicability and the resulting cost. In the next section and then in \Cref{sec:algo} we quantify instead how to choose $\ell, \la_1, \la_2$ to achieve $|\hOT - \OT(\mu, \nu)| \leq \eps$ with high probability and we provide a complete computational complexity in $\eps$.

\subsection{Theoretical Guarantees}
Here we quantify the convergence rate of $\hOT$ to $\OT$. To simplify the exposition, in this section we will make a classical assumption on the smoothness of the densities \citep{philippis2013mongeampre}. Note however that the results of the paper hold under a more general assumption on the smoothness of the potentials (see \cref{thm:SoS-OT}).
\begin{assumption}[$m$-times differentiable densities]\label{assum:measures}
Let $m, d \in \NN$. Let $\mu, \nu \in {\cal P}(\R^d)$.
\begin{enumerate}[a)]
    \item \label{assum:support_sets} $\mu, \nu$ have densities. Their supports, resp. $X, Y \subset \RR^d$ are convex, bounded and open;
    \item \label{assum:smooth_densties} the densities are finite, bounded away from zero, with Lipschitz  derivatives up to order $m$.
\end{enumerate}
\end{assumption}
\cref{assum:measures} is particularly adapted to our context since it guarantees that the Kantorovich potentials have a similar order of differentiability \citep{philippis2013mongeampre}.
The main result, \Cref{thm:sample_bound}, is expressed for a general set of couples $(\tilde{x}_j, \tilde{y}_j)$, $j \in [\ell]$. Here, we specify it for the case where the couples are sampled independently and uniformly at random.

\begin{theorem}\label{thm:uniform}
Let $\mu, \nu$ satisfy \cref{assum:measures} for $m > d$ and $m \geq 3$. Let $\ell \in \N$ and $\delta \in (0,1]$. Let $(\tilde{x}_j, \tilde{y}_j)$ be independently sampled from the uniform distribution on $X \times Y$. 
Let $\hOT$ be computed with $k_X = k_{m+1}$, $k_Y = k_{m+1}$ and $k_{XY} = k_m$ where $k_s$ for $s > 0$ is the Sobolev kernel in \cref{eq:sobolev-kernel}. 
Then, there exists $\ell_0$ depending only on $d,m,X,Y$ and $C_1, C_2$ depending only on $u_\star, v_\star$ and $d$, such that when $\ell \geq \ell_0$ and $\la_1, \la_2$ are chosen to satisfy
\begin{equation}\label{eq:choose-la1-la2}
    \la_1 \geq C_1 \ell^{-m/2d+1/2} \log \tfrac{\ell}{\delta}, \quad \la_2 \geq \|w_\mu - \hat{w}_\mu\|_\hhx + \|w_\nu - \hat{w}_\nu\|_\hhy + \la_1,
\end{equation}
then, with probability $1-\delta$, we have
$$ |\hOT - \OT(\mu, \nu)| \leq C_2 \la_2.$$
\end{theorem}
Note that while the rate does not depend exponentially in $d$ as we will see in the rest of the section, the constants $\ell_0, C_1, C_2$ depend exponentially in $d$ in the worst case, \BM{I changed the text in \Cref{thm:uniform} because we were saying the $C_1$ and $C_2$ only depend on $u_\star, v_\star$ and not $d$} as \cite{rudi2020global} for the case of global optimization. From the theorem above it is clear that the approximation error of $\hOT$ is the sum of the error induced by the kernel mean embeddings plus the error induced by the subsampling of the inequality. Note here that the result of the theorem above holds also if the $\ell$ couples are i.i.d. from $\rho = \mu \otimes \nu$, as discussed in Remark~\ref{rem:sampling-munu}. This can be beneficial if we do not know $X, Y$ or we do not know how to sample from them. In the next corollary we will specialize the result depending on the considered scenarios.

\begin{corollary}\label{cor:costs}
Under the same assumptions as \cref{thm:uniform}, let $k_X = k_{m+1}$, $k_Y = k_{m+1}, k_{XY} = k_m$ where $k_s$ for $s > 0$ is defined in \cref{eq:sobolev-kernel} and $\la_1 \geq C_1 \ell^{-(m-d)/2d} \log \tfrac{\ell}{\delta}$. Compute $\widehat{OT}$ with $\hat{w}_\mu, \hat{w}_\nu$ chosen according to one of the three scenarios below, as in \cref{sec:mean_estimators}. There exist $C, C', C_2', C_2''$ s.t. with probability at least $1-\delta$, \BM{added the "with probability..." part}
\begin{enumerate}
\item (Exact integral) When we are able to compute exactly $\int k_X(x,x') d\mu(x') d\mu(x)$ and also $\int_X k_X(x, x') d\mu(x)$ for any $x' \in X$ (and analogously for $\nu$). Choose $\la_2 = \la_1$. Then,
$$|\hOT - \OT(\mu, \nu)| \leq C_2 \ell^{-(m-d)/2d} \log \tfrac{\ell}{\delta}.$$
\item (Evaluation) When we are only able to evaluate $\mu, \nu$ on given points and to compute $\int_X k_X(x,z)k_X(x',z) dz, \int_X \int_X k_X(x,z)k_X(z,z')k_X(x',z') dz dz'$. Evaluate $\mu$ in $n_\mu$ points sampled uniformly from $X$ (and $n_\nu$ for $\nu$). Let $\la_2 = \la_1 + C(n_\mu + n_\nu)^{-(m+1)/d}\log \tfrac{n_\mu+ n_\nu}{\delta}$,
$$|\hOT - \OT(\mu, \nu)| \leq C_2' (n_\mu^{-(m+1)/d}\log \tfrac{n_\mu}{\delta} + n_\nu^{-(m+1)/d}\log \tfrac{n_\nu}{\delta} +  \ell^{-(m-d)/2d} \log \tfrac{\ell}{\delta}).$$
\item (Sampling) When we are only able to sample from $\mu, \nu$, by using $n_\mu$ i.i.d.~samples from $\mu$ and $n_\nu$ from $\nu$. Choose $\la_2 = \la_1 + C'(n_\mu + n_\nu)^{-1/2}\log \tfrac{n_\mu+n_\nu}{\delta} $. Then,
$$|\hOT - \OT(\mu, \nu)| \leq C_2'' (n_\mu^{-1/2}\log \tfrac{n_\mu}{\delta} + n_\nu^{-1/2}\log \tfrac{n_\nu}{\delta} +  \ell^{-(m-d)/2d} \log \tfrac{\ell}{\delta}).$$
\end{enumerate}
\end{corollary}

\section{Notations and background}\label{sec:background}
%
Let $n \in \NN$, we denote by $[n]$ the set $\{1,\dots,n\}$. For a set $Z$, and a positive definite kernel $k:Z \times Z \to \RR$ (i.e., so that all matrices of pairwise kernel evaluations are positive semi-definite), we can define a {\em reproducing kernel Hilbert spaces} \citep{aronszajn1950theory} $\hh$   of real functions on $Z$, endowed with an inner product $\scal{\cdot}{\cdot}_\hh$, and a norm  $\|\cdot\|_\hh$. It satisfies: (1) $k(z, \cdot) \in \hh$ for any $z \in Z$ and (2) the {\em reproducing property}, i.e., for any $f \in \hh, z \in Z$ it holds that $f(z) = \scal{f}{k(z,\cdot)}_\hh$. The {\em canonical feature map} associated to $\hh$ is the map $\phi:Z \to \hh$ corresponding to $z \mapsto k(z,\cdot)$, so that $k(z,z') = \scal{\phi(z)}{\phi(z')}_\hh$ \citep{aronszajn1950theory}.  

In this paper we use Sobolev spaces, defined on $Z \subseteq \R^d$, with $d \in \N$, an open set. For  $s \in \N$, denote by $H^s(Z)$ the {\em Sobolev space} of functions whose weak derivatives up to order $s$ are square-integrable, i.e., $H^s(Z) := \{f \in L^2(Z) ~|~ \|f\|_{H^s(Z)} < \infty \}$ and $\|f\|_{H^s(Z)} := \sum_{|\alpha| \leq s} \|D^\alpha f\|_{L^2(Z)}$ \citep{adams2003sobolev}. A remarkable property of $H^s(Z)$ that we are going to use in the proofs is that $H^s(Z) \subset C^{k}(Z)$ for any $s > d/2 + k$ and $k \in \N$. Moreover $H^{m+1}(Z) \subset H^m(Z), \forall m \in \N$.

\begin{proposition}[Sobolev kernel, \cite{wendland2004scattered}]\label{ex:sobolev-kernel} 
Let $Z \subset \R^d$, $d \in \N$, be an open bounded set. Let $s > d/2$. Define
\begin{equation}\label{eq:sobolev-kernel}
k_s(z,z') = c_s \|z-z'\|^{s-d/2} {\cal K}_{s-d/2}(\|z-z'\|), \quad \forall z,z' \in Z,
\end{equation}
where ${\cal K}:\R \to \R$ is the Bessel function of the second kind (see, e.g., Eq. 5.10 of the same book) and $c_s = \frac{2^{1 + d/2 -s}}{\Gamma(s-d/2)}$. Then the function $k_s$ is a kernel. Denoting by $\hh_Z$ the associated RKHS, when $Z$ has Lipschitz boundary, then $\hh_Z = H^s(Z)$ and the norms are equivalent.

\end{proposition}
In the particular case of $s = d/2 + 1/2$, we have $k_s(z,z') = \exp( - \| z - z'\|)$.  Note that the constant $c_s$ is chosen such that $k_s(z,z) = 1$ for any $z \in Z$.

\section{Positive operator representation for Kantorovich potentials}\label{sec:sos-rep}

We start with the following representation result on the structure of the optimal potentials, which is one of our main contributions in this paper.

\begin{theorem}
\label{thm:SoS-OT}
    Let $\mu \in \mathcal{P}(X), \nu \in \mathcal{P}(Y)$ satisfying \Cref{assum:measures}\ref{assum:support_sets} and let $(u_\star,v_\star)$ be Kantorovich potentials such that $u_\star \in H^{s+2}(X)$ and $v_\star \in H^{s+2}(Y)$ for $s > d+1$. There exist functions $w_1, ..., w_d \in H^{s}(X \times Y)$ such that
    $$\tfrac{1}{2}\|x-y\|^2- u_\star(x) - v_\star(y) = \textstyle \sum_{i=1}^d w_i(x,y)^2, \quad \forall (x,y) \in X \times Y.$$
\end{theorem}
\begin{proof}
Denote by $h$ the function $h(x,y) = c(x,y) - u_\star(x) - v_\star(y)$ for all $(x,y) \in X \times Y$. 
Let $f(x) = \frac{1}{2} \| x\|^2 - u_\star(x),  x \in X$. By Brenier's theoremfor quadratic optimal transport~\citep[Theorem 1.22]{brenier1987decomposition,santambrogio2015optimal},
\begin{enumerate}[(i)]
    \item\label{item:T_grad} $T = \nabla f$, where $T$ is the optimal transport map from $\mu$ to $\nu$,
    \item\label{item:psi_convex} $f$ is convex on $X$,
    \item\label{item:fenchel_young_decomp} $h$ is characterized by $h(x, y) = f(x) + f^\star(y) - x^\top y$, where $f^\star : y \in Y \mapsto \sup_{x\in X} {x}^\top{y} - f(x)$ is the Fenchel-Legendre conjugate of $f$. Moreover, $f^\star(y) = \frac{1}{2}\|y^2\| - v_\star(y)$.
 \end{enumerate}
%
%
Further, from the properties of Fenchel-Legendre conjugacy~\citep[Section 26]{rockafellar}, we have $T^{-1} = \nabla f^\star$. Hence, since $u_\star \in H^{s+2}(X)$ and $v_\star \in H^{s+2}(Y)$, we have

\begin{enumerate}[(i)]
    \setcounter{enumi}{3}
     \item\label{item:T_regular} $T = \nabla f$ (resp.\ $T^{-1} = \nabla f^\star$) is a $H^{s+1}$-diffeomorphism from $\overline{X}$ to $\overline{Y}$ (resp.\ $\overline{Y}$ to $\overline{X}$).
 \end{enumerate}
Since $f\in H^{s+2}(X)$ and $s > d/2$ and $X$ is a bounded open set with locally Lipschitz boundary (see Lemma~\ref{lm:uic-convex-set}), we have $H^{s+2}(X) \subset C^2(X)$ \citep{adams2003sobolev} and the Hessian ${\bf H}_f(x)$ is well defined for any $x \in X$.
Since, by \eqref{item:T_regular} , $T = \nabla f$ is a diffeomorphism, then, by Fenchel-Legendre conjugacy, $f$ is strictly convex~\citep{rockafellar}. Hence by compactness of~$\overline{X}$, $f$ has a Hessian~${\bf H}_f(x)$ which is bounded away from $0$. This implies: 
\begin{enumerate}[(i)]
    \setcounter{enumi}{4}   
     \item \label{item:bounded_hessian} There exists $\rho > 0$ such that ${\bf H}_f(x) \succeq \rho \Id$ for all $x \in X$.
\end{enumerate}
We will now use the decomposition \eqref{item:fenchel_young_decomp} along with a reparameterization of $h$ to obtain a decomposition as a sum of squares. Let 
\begin{equation}
   \tilde{h}(x, z) \defeq h(x, T(z)), \quad \forall (x, z) \in X \times X.
\end{equation}
The effect of this change of coordinates is illustrated in \Cref{fig:constraint_function}.
\begin{figure}[t]
    \centering
        \includegraphics[width = .49\textwidth]{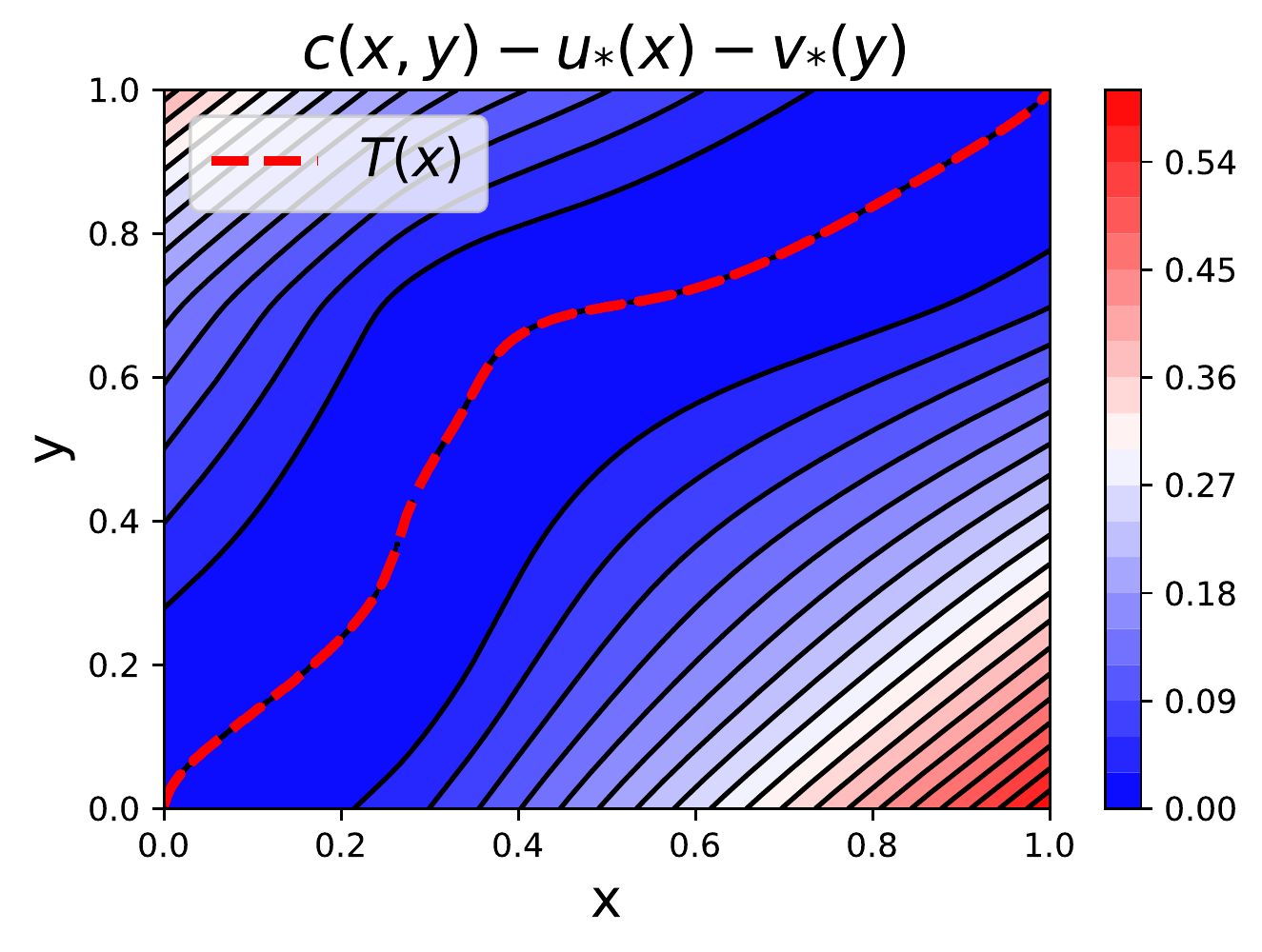}
        \includegraphics[width = .49\textwidth]{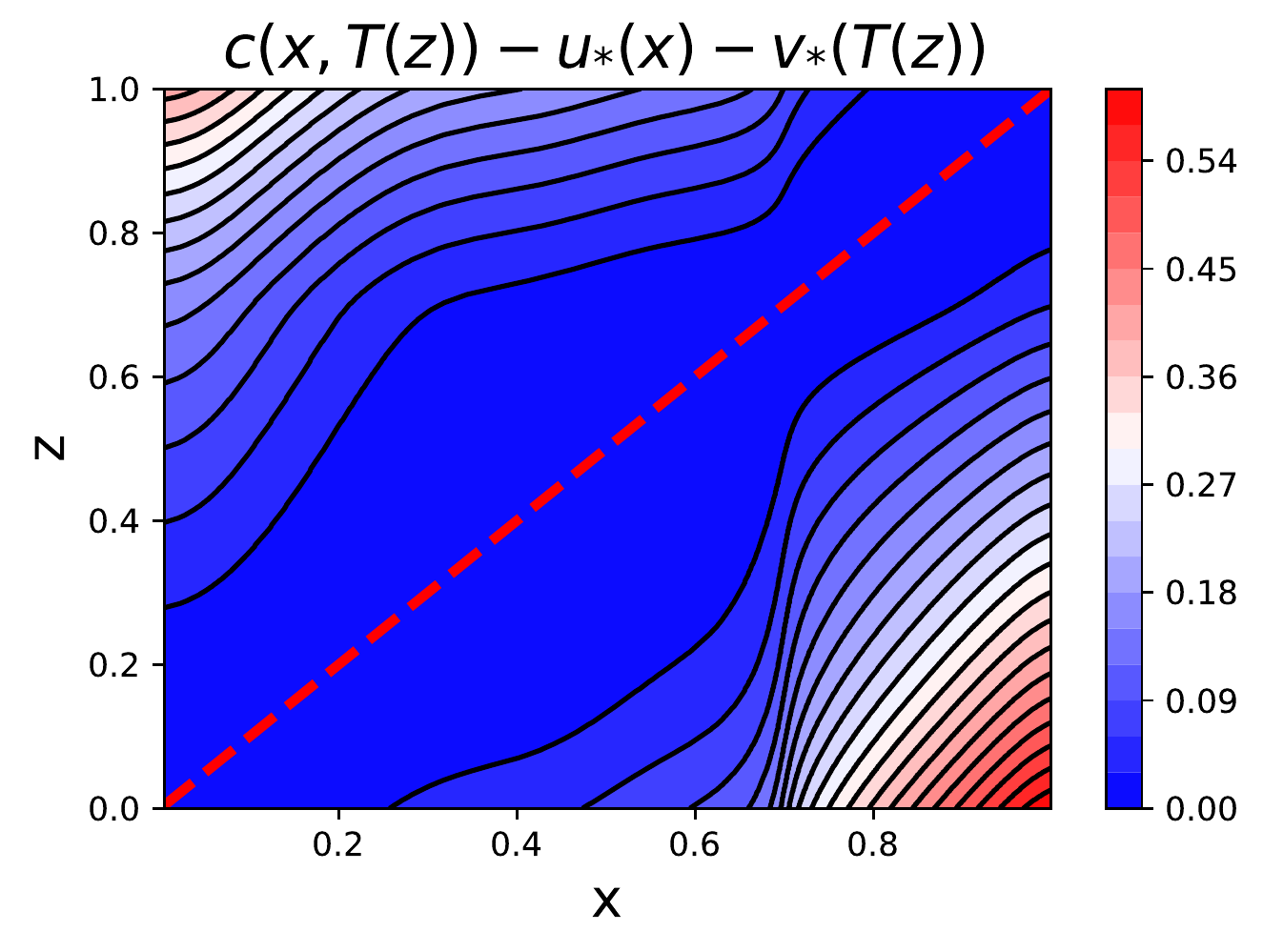}
        
        \vspace*{-.5cm}
        
    \caption{\textit{Left}: dual constraint function $h(x,y)$ corresponding to the measures in \Cref{fig:transport_map_fill,fig:transport_map_fill} in \Cref{sec:numerical}. Note that $h$ attains its minimum on the graph of transportation map $T$, and is elsewhere positive. \textit{Right}: changing parameterization flattens $h$ in the neighborhood of the graph of $T$. The red dotted line represent the zeros of both functions, and coincides with the graph of $T$ in the original parameterization (\textit{left}).}
    \label{fig:constraint_function}
\end{figure}
Since $f$ is differentiable, by the properties of the Fenchel-Legendre conjugate it holds that $f^\star(\nabla f(z)) = \nabla f(z)^\top z - f(z)$ for any $z \in X$~\citep{rockafellar}. Therefore, from \eqref{item:T_grad} we have that
\begin{equation}
    \tilde{h}(x, z) = f(x) - f(z) - \nabla f(z)^\top(x - z).
\end{equation}
%
Now, since $X$ is convex, we can characterize $f$ in terms of its Taylor expansion:
$$f(x) = f(z) + \nabla f(z)^\top(x - z) + (x - z)^\top \bR(x, z) (x - z), \quad \forall x,z \in X,$$
where $\bR$ is the integral reminder $\bR(x, z) \defeq \int_0^1 (1-t) \bH_f((1 - t)x + tz) \d t$. This implies
\begin{equation} 
    \tilde{h}(x, z) = (x - z)^\top \bR(x, z) (x - z), \quad \forall x,z \in X.
\end{equation}
From \eqref{item:bounded_hessian}, we get $\forall x, z, \bR(x, z) \succeq \frac{\rho}{2} \Id$. In particular, for all $x, z \in X$, the matrix $\bR(x, z)$ admits a positive square root $\sqrt{\bR(x, z)}$. Further, since $\sqrt{\cdot}$ is $C^\infty$ on the closed set $\{\bA \in \mathbb{S}_+(\Rd) : \bA \succeq \tfrac{\rho}{2} \Id\}$ and $\frac{\partial^2}{\partial x_i\partial x_j} f \in H^s(X)$ for all $i, j \in [d]$, the functions $r_{i,j}: (x, z) \mapsto e_i^\top\sqrt{\bR(x, z)}e_j$ are in $H^s(X\times X)$ for all $i, j \in [d]$ \citep[see Proposition 1 and Assumption 2b of ][]{rudi2020global},  where $(e_1, ..., e_d)$ is the canonical ONB of $\RR^d$. 
Define now the functions $$\tilde{w}_i(x, z) \defeq \sum_{j=1}^d r_{i,j}(x, z) (e_j^\top(x-z)), \quad  \forall x, z \in X, i \in [d].$$ From the above arguments, it holds that  $\tilde{w}_i\in H^s(X\times X), i\in [d]$, and $\tilde{h}(x, z) = \sum_{i=1}^d \tilde{w}^2_i(x, z)$.
Now, since $T$ is a $H^{s+1}$-diffeomorphism from $X$ to $Y$, changing parameterization again we have
$h(x, y) = \sum_{i=1}^d w^2_i(x, y), ~ \forall (x, y) \in X \times Y$,
with $w_i(x,y) = \tilde{w}_i(x, T^{-1}(y)), ~ \forall (x, y) \in X \times Y$.\vspace{0.2cm}

\noindent We conclude by proving that $w_i \in H^s(X \times Y)$ for all $i \in [d]$. Indeed, from \eqref{item:T_regular} $T^{-1}$ is a $H^{s+1}$-diffeomorphism from $\bar{Y}$ to $\bar{X}$ and it is bi-Lipschitz (since $T$ and $T^{-1}$ are Lipschitz due to the continuity of their Hessian and the boundedness of $X, Y$). Define the map $Q$ as $(x,y) \mapsto (x, T^{-1}(y))$ and note that $Q \in H^{s+1}(X \times Y, \R^{2d})$, by construction. Note that $\tilde{w}_i \in H^{s}(X \times X)$ and has bounded weak derivatives of order $1$, since $s > d+1$ and $H^s(X \times X)$ is bounded \citep{adams2003sobolev}. The conditions above on $\tilde{w}_i, Q$ guarantee that $w_i = \tilde{w}_i \circ Q$ belongs to $H^s(X \times Y)$  \citep[see, e.g., Theorem 1.2 of][]{campbell2015weak}.
\end{proof}
\cref{thm:SoS-OT} implies the existence of $A_\star \in \pdm{\hhxy}$ by effect of the reproducing property, when we consider a RKHS $\hhxy$ containing $H^s(X\times Y)$. The proof is in \cref{sect:proof-Astar}, \cpageref{sect:proof-Astar}.
\begin{corollary}\label{cor:Astar_exists}
    \label{prop:representation}
        Let $\hhxy$ be a RKHS such that $H^s(X\times Y) \subseteq \hhxy$. Under the hypothesis of \Cref{thm:SoS-OT}, there exists a positive operator $A_\star \in \pdm{\hhxy}$ with rank at most $d$, such that 
    \begin{equation}\label{eq:Astar-equality}
          c(x,y) - u_\star(x) - v_\star(y) = \scal{\phi(x,y)}{A_{*} \phi(x,y)}_\hhxy.
    \end{equation}
\end{corollary}
Finally, the following corollary shows the effect of \cref{assum:measures} on the existence of $A_\star$. Note indeed that such an assumption implies smoothness of the Kantorovich potentials \citep{philippis2013mongeampre}. If $m > d$ they are smooth enough to apply \cref{thm:SoS-OT} and the corollary above. The proof of the following corollary is in \cref{sect:proof-Astar}, \cpageref{sect:proof-Astar}.
\begin{corollary}[Effects of Asm.~\ref{assum:measures}]\label{cor:asm1-Astar}
Let $\mu, \nu$ satisfy \cref{assum:measures} for $m > d$. Let $(u_\star, v_\star)$ be Kantorovich potentials for $\mu, \nu$. Let $\hhx = H^{m+3}(X), \hhy = H^{m+3}(Y), \hhxy = H^{m}(X \times Y)$. Then $u_\star \in \hhx, v_\star \in \hhy$ and there exists $A_\star \in \pdm{\hhxy}$ satifying \cref{eq:Astar-equality} and $\operatorname{rank} A_\star \leq d$.
\end{corollary}

\section{Subsampling the constraints}\label{sec:sample_bounds}

Given $\ell \in \NN$ and a set of points $\tilde{Z}_\ell = \{(\tilde{x}_1, \tilde{y}_1),\dots, (\tilde{x}_\ell, \tilde{y}_\ell)\}$, define the \emph{fill distance} \citep{wendland2004scattered},
\begin{equation}\label{eq:fill-distance}
h_\ell = \sup_{x \in X, y \in Y} \min_{j \in [\ell]} \|(x,y) - (\tilde{x}_j, \tilde{y}_j)\|.
\end{equation}
The following theorem, that is an adaptation of Theorem 4 from \cite{rudi2020global}, quantifies the error of subsampling the constraints in terms of the fill distance.
\begin{theorem}\label{thm:eps-exists}
Let $X, Y$ satisfy \Cref{assum:measures}\ref{assum:support_sets}. Let $s \geq 3$ and $s > d$. Let $\hhx \subseteq H^s(X), \hhy \subseteq H^s(Y), \hhxy = H^s(X \times Y)$. Let $\tilde{Z}_\ell \subset X \times Y$ be a set of points of cardinality $\ell$ and fill distance $h_\ell$ and let $u \in \hhx, v \in \hhy, A \in \pdm{\hhxy}$ satisfy
\begin{equation}\label{eq:constr-ineq}
    c(\tilde{x}_j,\tilde{y}_j) - u(\tilde{x}_j) - v(\tilde{y}_j) = \scal{\phi(\tilde{x}_j,\tilde{y}_j)}{A \phi(\tilde{x}_j, \tilde{y}_j)}_\hhxy, \quad \forall j \in [\ell].
\end{equation}
There exist $h_0, C_0$ depending only on $s, d, X, Y$ such that, when $h_\ell \leq h_0$, then
$$c(x,y) \geq u(x) + v(y) - \eps, ~~~ \forall x,y \in X \times Y, \quad \textrm{where} \quad \eps = Q h^{s-d}_\ell,$$
where $Q = C_0 (\|u\|_\hhx + \|v\|_\hhy + \tr(A))$. Note that $h_0, C_0$ depend only on $d,m,X,Y$.
\end{theorem}
The proof of the theorem above is reported in \cref{sect:proof-subs}, \cpageref{sect:proof-subs}. Using the theorem above we are able to show that, given a maximizer $(\hat{u}, \hat{v}, \hat{A})$ of Problem~\ref{eq:w_hat} the couple $(\hat{u}-\eps/2, \hat{v}-\eps/2)$ is admissible for Problem~\ref{EqDualOT}. This is a crucial step to bound $|\hOT - \OT|$ in terms of the fill distance $h_\ell$ and it is stated next.
\begin{theorem}\label{thm:sample_bound}
Let $\mu, \nu$ and $X, Y \subset \R^d$ satisfy \cref{assum:measures}\ref{assum:support_sets}. Let $s > d$ and let $\hhx, \hhy, \hhxy$ be RKHS on $X,Y, X\times Y$ such that $\hhx \subseteq H^s(X), \hhy \subseteq H^s(Y), \hhxy = H^s(X\times Y)$. Assume that there exist two Kantorovich potentials $u_\star, v_\star$ such that $u_\star \in \hhx, v_\star \in \hhy$ and there exists $A_\star \in \pdm{\hhxy}$ that satisfies \cref{eq:Astar-equality}. Let $\hOT$ be computed according to \cref{eq:hat_OT} using a set of $\ell \in \N$ points $\tilde{Z}_\ell \subset X \times Y$ with fill distance $h_\ell$. Let $h_0, C_0$ as in \cref{thm:eps-exists}. Let
$$ \eta = C_0 h^{s-d}_\ell, \quad \gamma = \|w_\mu - \hat{w}_\mu\|_\hhx + \|w_\nu - \hat{w}_\nu\|_\hhy.$$
When $h_\ell \leq h_0$, $\la_2 > 0$ and $\la_1$ is chosen such that $\la_1 \geq 2\eta$, then
$$|\hOT - \OT| ~\leq~ 6\la_1\tr(A_\star) ~+~ 6~\tfrac{(\gamma + \eta)^2}{\la_2} + 6~ \la_2~ \big(\|u_\star\|^2_\hhx + \|v_\star\|^2_\hhy).$$
\end{theorem}
The proof of the theorem above is in \cref{sect:proof-subs}, \cpageref{sect:proof-subs}. \cref{thm:sample_bound} together with \cref{thm:SoS-OT} (in particular Corollary~\ref{cor:asm1-Astar}) allow to prove  \cref{thm:uniform}. To bound $h_\ell$ in terms of $\ell$, we used a result recalled in Lemma~\ref{lm:uic-sampling}.\vspace{0.2cm}
\begin{proof}{\bf of \cref{thm:uniform}.} First note that $k_X = k_{m+1}, k_Y = k_{m+1}, k_{XY} = k_m$ imply that $\hhx = H^{m+1}(X), \hhy = H^{m+1}(X), \hhxy = H^{m}(X \times Y)$. Then, by Corollary~\ref{cor:asm1-Astar} we have that under \cref{assum:measures} for any Kantorovich potentials $u_\star, v_\star$ there exists a finite rank $A_\star \in \SS_+(\hhxy)$ such that $c(x,y) - u_\star(x) - v_\star(y) = \scal{\phi(x,y)}{A_{*} \phi(x,y)}_\hhxy$, and that $u_\star \in H^{m+3}(X) \subseteq H^{m+1}(X) = \hhx$ and analogously $v_\star \in \hhy$. Among them, we select the triplet minimizing $\|u_\star\|^2_{\hhx} + \|v_\star\|^2_{\hhy} + \tr(A_\star)$ and we denote by $Q_{\mu, \nu}$ the resulting minimum. The proof is obtained by using this triplet in Theorem~\ref{thm:sample_bound} and bounding $h_\ell$ as follows. First note that $X \times Y$ is a convex bounded set, since $X, Y$ have the same property. As recalled in Lemma~\ref{lm:uic-convex-set}, \cpageref{lm:uic-convex-set} of the appendix, convex bounded sets have the so-called {\em uniform interior cone condition}. This guarantees that $\ell$ i.i.d. points sampled from a distribution $\rho$, that has a density bounded away from zero, achieve the following bound on the fill distance: $h_\ell \leq (C \ell^{-1} \log(C' \ell/\delta))^{-1/(2d)}$, with probability at least $1-\delta$, when $\ell \geq \ell_0$. Here $\ell_0, C, C'$ are constants that depend only on $d, X \times Y$ and the constant $c_0$ for which the density of $\rho$ is bounded away from zero. The final constants $C_1, C_2$ depend also on $Q_{\mu,\nu}$. 
\end{proof}
\noindent We conclude with the following remark that is useful when we do not know the supports $X, Y$ or we are not able to sample i.i.d.~points from the uniform distribution on them. 
\begin{remark}[Sampling from $X \times Y$ using $\mu \otimes \nu$]\label{rem:sampling-munu} Since, under \cref{assum:measures}, we have $\mu$ and $\nu$ bounded away from $0$, we can compute $\hOT$ by using $\ell$  i.i.d.~samples from $\rho = \mu \otimes \nu$, obtaining the same guarantees as Theorem~\ref{thm:uniform}. Indeed, by inspecting the proof above, it is only required that $\rho$ has support $X\times Y$, and has a density that is bounded away from $0$. However, the constants $\ell_0, C_1, C_2$ will depend also on how far the density of $\rho$ is bounded away from $0$. 
\end{remark}

\section{Estimators for the kernel mean embeddings of $\mu, \nu$}\label{sec:mean_estimators}
In this section we consider three classes of estimators $(\hat{w}_\mu, \hat{w}_\nu)$ for the kernel mean embeddings $w_\mu \in \hhx$ and $w_\nu \in \hhy$ defined as $w_\mu = \int \phi_X(x) d\mu(x)$ and $ w_\nu = \int \phi_Y(y) d\nu(y)$.
 As we observed in the introduction to \cref{eq:w_hat_dual}, the operations we need to perform on $\hat{w}_\mu, \hat{w}_\nu$ to compute the algorithm are the evaluation of the norm $\|\hat{w}_\mu\|^2_\hhx$ and the evaluation of $\hat{w}_\mu(\tilde{x}_i)$ for $i \in [\ell]$ (and the same for $\hat{w}_\nu$). For each class we will specify such operations. Here, we assume that $\phi_X, \phi_Y$ are uniformly bounded maps (as for the Sobolev kernel, where $\sup_{x \in X} \|\phi_X(x)\|^2_\hhx = \sup_{x \in X} k(x,x) = 1$, see Proposition~\ref{ex:sobolev-kernel}).  
Clearly a class of estimators must only be chosen if we are able to perform the required operations. 
 
\paragraph{Exact integral.} Here we take $\hat{w}_\mu := w_\mu$ and $\hat{w}_\nu = w_\nu$ and we report only the operations for $\mu$ since the ones for $\nu$ are the same. This is the estimator that leads to the best rates as shown in \Cref{thm:complexities} and Corollary~\ref{cor:costs}. However it requires to perform the most difficult operations, i.e., (1) $\hat{w}_\mu(\tilde{x}_i) = \scal{w_\mu}{\phi_X(\tilde{x}_i)}_\hhx = \int_X \scal{\phi_X(\tilde{x}_i)}{\phi_X(x)}d\mu(x) = \int_X k_X(\tilde{x}_i, x) d\mu(x)$ for all $i \in [\ell]$; (2) $\|\hat{w}_\mu\|^2_\hhx = \scal{w_\mu}{w_\mu}_\hhx = \int_X \scal{\phi_X(x)}{\phi_X(x')}_\hhx d\mu(x) d\mu(x') = \int_X k_X(x,x') d\mu(x) d\mu(x')$.
Moreover the costs $C, E$ in \cref{eq:cost} are $C = O(1), E = O(1)$.

\paragraph{Evaluation estimator.} Here we assume we are able to evaluate $\mu(x_j)$ in a given set of points $x_1, \dots, x_{n_\mu}$ with $n_\mu \in \N$ (analogously for $\nu$ on $y_1,\dots, y_{n_\nu}$ with $n_\nu \in \N$). We define the estimators as $\hat{w}_\mu = \int_X \phi_X(x) \hat{g}_\mu(x) dx$ (analogously for $\hat{w}_\nu$). Here $\hat{g}_\mu \in \hhx$ is the {\em kernel least squares} estimator \citep{narcowich2005sobolev} of the density $\mu$ defined as $\hat{g}_\mu(x) = \sum_{j \in [n_\mu]} \alpha_j k_X(x_j, x)$ where $\alpha = {\bf K}_X^{-1} c_\mu$ and ${\bf K}_X \in \RR^{n_\mu \times n_\mu}$,  $({\bf K}_X)_{i,j} = k_X(x_i, x_j)$ for all $i,j \in [n_\mu]$, while $c_\mu \in \R^{n_\mu}$, $c_\mu = (\mu(x_1), \dots, \mu(x_n))$. In this case the steps are: 
(1) $\hat{w}_\mu(\tilde{x}_i) = \int_X \scal{\phi_X(\tilde{x}_i)}{\phi_X(x)}\hat{g}_\mu(x)dx =  \sum_{j \in [n_\mu]} \alpha_j \int_X k_X(\tilde{x}_i,x)k_X(x_j, x) dx,  i \in [\ell]$, and (2)  $\|\hat{w}_\mu\|^2_\hhx = \sum_{i,t \in [n_\mu]} \alpha_i \alpha_t \int_X \int_X k_X(x_i,x)k(x,x')k(x_t,x') dx dx'$.
Moreover the costs $C, E$ in \cref{eq:cost} are $C = O(n_\mu^3+n_\nu^3), E = O(n_\mu + n_\nu)$.

\paragraph{Sample estimator.} Here we assume we are able to sample from $\mu, \nu$. Let 
$x_1,\dots, x_{n_\mu}$, $n_\mu \in \N$, be sampled i.i.d.~from $\mu$ and $y_1,\dots, y_{n_\nu}$, $n_\nu \in \N$  be sampled i.i.d.~from $\nu$. The estimators are $\hat{w}_\mu = \frac{1}{n_\mu} \sum_{j \in [n_\mu]} \phi_X(x_j)$, and analogously for $\hat{w}_\nu$. In this case the operations are:  (1) $\hat{w}_\mu(\tilde{x}_i) = \frac{1}{n_\mu} \sum_{j \in [n_\mu]} k_X(x_j, \tilde{x}_i)$, and (2) $\|\hat{w}_\mu\|^2_\hhx = \frac{1}{n_\mu^2} \sum_{i,j \in [n_\mu]} k_X(x_i, x_j)$. 
Moreover the costs $C, E$ in \cref{eq:cost} are $C = O(n_\mu^2+n_\nu^2), E = O(n_\mu + n_\nu)$.
\\

\noindent The effects of the estimators above are studied in \cref{thm:complexities} and Corollary~\ref{cor:costs}, reported in \cref{sec:intro} and proven in \cref{sect:kme}, \cpageref{sect:kme}, using standard tools from approximation theory and machine learning with kernel methods \citep{wendland2004scattered,caponnetto2007optimal,muandet2016kernel}.

\paragraph{Acknowledgements.}
This work was funded in part by the French government under management
of Agence Nationale de la Recherche as part of the “Investissements d’avenir” program, reference
ANR-19-P3IA-0001 (PRAIRIE 3IA Institute). We also acknowledge support from the European
Research Council (grants SEQUOIA 724063 and REAL 947908), and R{\'e}gion Ile-de-France.

\bibliography{biblio}

\appendix

\section{Proofs of Corollary~\ref{cor:Astar_exists} and Corollary~\ref{cor:asm1-Astar}}
\label{sect:proof-Astar}

\begin{proof}{\bf of Corollary~\ref{cor:Astar_exists}.}
Define $h(x, y) := c(x,y) - u_\star(x) - v_\star(y),  \forall (x, y) \in X \times Y$. By \Cref{thm:SoS-OT}, there exist $w_1, ..., w_d \in H^s(X \times Y)$ such that $h = \sum_{i=1}^d w_i^2$.
Since $H^s(X\times Y) \subseteq \hhxy$, then $w_1,\dots,w_d \in \hhxy$. Hence $A_\star = \sum_{i=1}^d w_i \otimes w_i \in \pdm{\hhxy}$. Moreover, by the reproducing property of $\hhxy$, $A_\star$ satisfies $\dotp{\phi(x, y)}{A_\star\phi(x, y)}_\hhxy = \sum_{i = 1}^d \scal{w_j}{\phi(x,y)}_\hhxy^2 = \sum_{i = 1}^d  w(x,y)^2 = h(x,y)$, for all $(x,y) \in X \times Y$. Finally, note that $\operatorname{rank} A_\star \leq d$ by construction.
\end{proof}

\begin{proof}{\bf of Corollary~\ref{cor:asm1-Astar}.}
Theorem 3.3 in \cite{philippis2013mongeampre} implies that Kantorovich potentials satisfy $u_\star \in C^{m+2,1}(\overline{X}), v_\star \in C^{m+2,1}(\overline{Y})$, where $C^{m+2,1}(\overline{Z})$ for an open set $Z$ is the space of real functions over $Z$ that are $m+2$-times differentiable, with all the derivatives of order $m+2$ that are Lipschitz continuous \citep[][1.29, Page~10]{adams2003sobolev}. Since $X$ is convex and bounded, then Lipschitz continuity of all the derivatives of order $m+2$ implies that all the weak derivatives up to order $m+3$ are in $L^\infty(X)$. Since $L^\infty(X) \subset L^2(X)$, by the boundedness of $X$, we have $C^{m+2,1}(\overline{X}) \subset H^{m+3}_2(X)$. Analogously $C^{m+2,1}(\overline{Y}) \subset H^{m+3}_2(Y)$. Then $u_\star \in H^{m+3}(X), v_\star \in H^{m+3}(Y)$, and we can apply \cref{thm:SoS-OT} and Corollary~\ref{cor:Astar_exists} with $s = m+1$, when $m > d$, which guarantee the existence of $A_\star \in \pdm{\hhxy}$, since $\hhxy = H^s(X \times Y)$.
\end{proof}

\section{Proofs of \cref{thm:eps-exists} and \cref{thm:sample_bound}}\label{sect:proof-subs}

\begin{proof}{\bf of \cref{thm:eps-exists}.}
Let $f(x,y) = c(x,y) - u(x){\bf 1}_X(y) - v(y){\bf 1}_Y(x) - \scal{\phi(x,y)}{A \phi(x,y)}$ for any $x, y \in X \times Y$, where the function ${\bf 1}_X(x)$ and ${\bf 1}_Y(y)$ are respectively the constant function $1$ over $X$ and over $Y$. 
By construction $f \in H^s(X \times Y)$. Since $X, Y$ are open bounded convex sets, then $X \times Y$ has the same property which, in turn, implies the so-called {\em uniform interior cone condition}  (see Lemma~\ref{lm:uic-convex-set}, \cpageref{lm:uic-convex-set} in the appendix). Then we can apply Theorem 1 of \cite{narcowich2005sobolev}  for which there exist $h_0, C_0$ depending only on $s, d, X, Y$ such that for any $h_\ell \leq h_0$ the following holds
$$\sup_{(x,y) \in X \times Y} |f(x,y)| \leq \eps, \quad \eps := C_1 h_\ell^{s-d} |f|_{H^s(X\times Y)},$$
where $|g|_{H^s(X\times Y)} := \sum_{|\alpha| = s} \|D^\alpha g\|_{L^2(X\times Y)} \leq \|g\|_{H^s(X\times Y)}$ for any $g \in H^s(X \times Y)$. Let $r_A(x,y) = \scal{\phi(\tilde{x}_j,\tilde{y}_j)}{A \phi(\tilde{x}_j, \tilde{y}_j)}_\hhxy$, we have 
$$|f|_{H^s(X\times Y)} \leq |c|_{H^s(X\times Y)} + |u {\bf 1}_X|_{H^s(X\times Y)} + |v {\bf 1}_Y|_{H^s(X\times Y)} +|r_A|_{H^s(X\times Y)}.$$
Now $|c|_{H^s(X\times Y)} = 0$ since $c$ is the quadratic cost and $s \geq 3$ and $|u {\bf 1}_X|_{H^s(X\times Y)} = |u|_{H^s(X)} \leq  \|u\|_{H^s(X)}$, $|v {\bf 1}_Y|_{H^s(X\times Y)} = |v|_{H^s(Y)} \leq  \|v\|_{H^s(Y)}$.
We recall now that for any two Banach spaces satisfying $B \subseteq A$, we have $\|w\|_A \leq \|w\|_B$ for any $w \in B$. Then $\|u\|_{H^s(X)} \leq \|u\|_{\hhx}$, $\|v\|_{H^s(Y)} \leq \|v\|_{\hhy}$. Finally $|r_A|_{H^s(X\times Y)} \leq \|r_A\|_{H^s(X \times Y)} \leq C_1 \tr(A)$ via Lemma~9 page 41 from \cite{rudi2020global} and Proposition 1 of the same paper, where $C_2$ depends only on $s,X,Y$. 
Now since $|f(x,y)| \leq \eps$ for all $(x,y) \in X \times Y$, and since $r_A(x,y) \geq 0$ for any $x,y \in X \times Y$ since $A$ is a positive operator, we have
$$-\eps \leq f(x,y) \leq f(x,y) + r_A(x,y) = c(x,y) - u(x) - v(y), \quad \forall (x,y) \in X \times Y.$$
The final result is obtained by defining $C_0 = C_1(1+C_2)$.
\end{proof}

\begin{proof}{\bf of \cref{thm:sample_bound}.}
Denote by $V(u,v)$ the functional of Problem~\ref{EqDualOT} and by $\hat{V}_{\la_1,\la_2}(u,v,A)$ the functional of Problem~\ref{eq:w_hat}. Then $\OT = V(u_\star, v_\star)$. Denote by $\Delta(u,v)$ the quantity 
$$\Delta(u,v) = |\scal{u}{\hat{w}_\mu - w_{\mu}}_\hhx + \scal{v}{\hat{w}_\nu - w_{\nu}}_\hhy|.$$
Denote by $\hat{R}^2(u,v)$ the quantity $\hat{R}^2(u,v) =  \|u\|^2_\hhx + \|v\|^2_\hhy$ for any $u \in \hhx, v \in \hhy$.

\paragraph{Step 0. Admissibility of $u_\star, v_\star, A_\star$ and existence of a maximizer.}
Note that $(u_\star, v_\star, A_\star)$ is an admissible point for Problem~\ref{eq:w_hat}, since the triple satisfies $c_\star(x,y) - u_\star(x) - v_\star(y) = \scal{\phi(x,y)}{A_\star \phi(x,y)} ~\forall (x,y) \in X \times Y$, and Problem~\ref{eq:w_hat} applies the same constraints but on a subset of $X \times Y$. 
Moreover $\la_1, \la_2 > 0$, this is enough to guarantee the existence of a maximizer for Problem~\ref{eq:w_hat}. Indeed, as we recall in Lemma~\ref{lm:prob_w_hat_has_minimizer}, \cpageref{lm:prob_w_hat_has_minimizer}, in the appendix, a form of {\em representer theorem} holds for Problem~\ref{eq:w_hat} (see Lemma~\ref{lm:representer}, \cpageref{lm:representer} in the appendix), moreover the functional is coercive on the finite-dimensional space induced by the data, it has an upper bound and the problem has an admissible point. 
Now denote by $(\hat{u}, \hat{v}, \hat{A})$ a maximizer of Problem~\ref{eq:w_hat} and define $\bOT := \hat{V}_{\la_1,\la_2}(\hat{u},\hat{v},\hat{A})$. By construction $\bOT$ is the maximum of $\hat{V}_{\la_1,\la_2}$.
Then by definition of $\hOT$ in \cref{eq:hat_OT},
$$\hOT = \bOT + \la_1\tr(\hat{A}) + \la_2 R^2(\hat{u},\hat{v}).$$ 

\paragraph{Step 1. Subsampling the inequality.}
The assumption on $h_\ell$ and the fact that the constraints of \cref{eq:w_hat} satisfy \cref{eq:constr-ineq} allow to apply \cref{thm:eps-exists}, from which there exists a couple $(\hat{u}_\eps, \hat{v}_\eps)$ that is admissible for Problem~\ref{EqDualOT}, where $\hat{u}_\eps = \hat{u} - \frac{\eps}{2}$ and $\hat{v}_\eps = \hat{v} - \frac{\eps}{2}$, for any $\eps \geq C_0 h^{s-d}_\ell (\|\hat{u}\|_\hhx + \|\hat{v}\|_\hhy +  \tr(\hat{A}))$. In particular, since $\|u\|_\hhx + \|v\|_\hhy \leq 2 R(\hat{u}, \hat{v})$, we choose $\eps = \eta (R(\hat{u}, \hat{v}) + \tr(\hat{A}))$, with $\eta = C_0 h^{s-d}_\ell$.

\paragraph{Step 2. Bounding $\hOT - \OT$.} Since $\OT$ is the maximum of Problem~\ref{EqDualOT} then
\begin{align}\label{eq:upperbound-tildeW-W}
     \OT = V(u_\star, v_\star) & \geq V(\hat{u}_\eps, \hat{v}_\eps) = V(\hat{u}, \hat{v}) - \eps \geq \hOT - \Delta(\hat{u}, \hat{v}) - \eps.
\end{align}
Analogously, since $(\hat{u}, \hat{v}, \hat{A})$ maximize Problem~\ref{eq:w_hat}, then
\begin{align}
\label{eq:bound-hatW-to-Vastar} \bOT  = \hat{V}_{\la_1, \la_2}(\hat{u}, \hat{v}, \hat{A}) & \geq \hat{V}_{\la_1,\la_2}(u_\star, v_\star, A_\star) \\
& = V(u_\star, v_\star) - \left[V(u_\star, v_\star) - \hat{V}_{\la_1,\la_2}(u_\star, v_\star, A_\star)\right]\\
& \geq \OT - \Delta(u_\star, v_\star) - \la_1\tr(A_\star) + \la_2 R^2(u_\star,v_\star).
\end{align}
By expressing $\bOT$ in terms of $\hOT$ in the inequality above and combining it with \cref{eq:upperbound-tildeW-W}, we have
\begin{equation}\label{eq:bound-Wtilde-W}
    \la_1\tr(\hat{A} - A_\star) +  \la_2 (R^2(\hat{u}, \hat{v}) - R^2(u_\star,v_\star)) - \Delta(u_\star, v_\star) \leq \hOT - \OT \leq \Delta(\hat{u}, \hat{v}) + \eps.
\end{equation}

\paragraph{Step 3. Bound on $\Delta$.}
Note that for two RKHS $\hh, \kk$ and any $u,v \in \hh, w,z \in \kk$, the following identity holds: $\scal{u}{v}_\hh + \scal{w}{z}_\kk = \scal{(u,w)}{(v,z)}_{\hh \oplus \kk}$, where $\hh \oplus \kk$ is a RKHS \citep{aronszajn1950theory}. This implies $(\scal{u}{v}_\hh + \scal{w}{z}_\kk)^2 = \|(u,w)\|_{\hh \oplus \kk}^2\|(v,z)\|_{\hh \oplus \kk}^2 = (\|u\|^2_\hh + \|v\|^2_\kk)(\|w\|^2_\hh + \|z\|^2_\kk)$. Applying this inequality to $\Delta$ leads to the following result:
\begin{align}\label{eq:bound-Delta}
\Delta(a,b) \leq R(a,b) \gamma.
\end{align}

\paragraph{Step 4. Bounding $\tr(\hat{A}), R(\hat{u}, \hat{v})$.}

The bound \cref{eq:bound-Wtilde-W} implies
\begin{align}
\la_1\tr(\hat{A}) +  \la_2 R^2(\hat{u}, \hat{v}) \leq \la_1\tr(A_\star) +  \la_2 R^2(u_\star,v_\star) + \Delta(u_\star, v_\star) + \Delta(\hat{u}, \hat{v}) + \eps.
\end{align}
By bounding $\Delta$ via \cref{eq:bound-Delta}, expanding the definition of $\eps$ and reordering the terms in the inequality above, we obtain
\begin{align}
\alpha \tr(\hat{A}) +  \la_2 R^2(\hat{u}, \hat{v}) -  \beta R(\hat{u}, \hat{v}) \leq \la_1\tr(A_\star) + \la_2 R^2(u_\star,v_\star) + \gamma R(u_\star, v_\star),
\end{align} 
with $\alpha = \la_1 - \eta$ and $\beta = \gamma + \eta$.  By completing the square in $R^2(\hat{u}, \hat{v})$, the inequality above is rewritten as
\begin{align}
\alpha \tr(\hat{A}) +  \la_2 (R(\hat{u}, \hat{v}) -  \tfrac{\beta}{2\la_2})^2 \leq \la_1\tr(A_\star) + \la_2 R^2(u_\star,v_\star) + \gamma R(u_\star, v_\star) + \tfrac{\beta^2}{4\la_2}.
\end{align} 
Since $\alpha \geq \la_1/2$, by the assumption $\la_1 \geq 2 \eta$, the inequality above implies
\begin{align}\label{eq:bound-trA-R}
\tfrac{\la_1}{2}\tr(\hat{A}) \leq \la_2 S,  \quad R(\hat{u}, \hat{v}) \leq \tfrac{\beta}{2\la_2} + \sqrt{S}.
\end{align} 
with $S  := \tfrac{\la_1}{\la_2}\tr(A_\star) + R^2(u_\star,v_\star) + \tfrac{\gamma }{\la_2} R(u_\star, v_\star) + \tfrac{\beta^2}{4\la_2^2}.$
\paragraph{Conclusion.} From the lower bound in \cref{eq:bound-Wtilde-W} and the bound \cref{eq:bound-Delta} on $\Delta$, we have that
$$\hOT - \OT  \geq -\la_1\tr(A_\star) - \la_2 R^2(u_\star,v_\star) - \gamma  R(u_\star, v_\star) \geq - \la_2 S.$$
From the upper bound in \cref{eq:bound-Wtilde-W}, the bound for $\Delta$ in \cref{eq:bound-Delta}, the bound for $\tr(\hat{A}), R(\hat{u},\hat{v})$ in \cref{eq:bound-trA-R}, the definition of $\eps$, and the fact that $\la_1 \geq 2 \eta$, we have 
\begin{align}
\hOT - \OT  &\leq \beta R(\hat{u}, \hat{v}) + \tfrac{\la_1}{2} \tr(\hat{A}) \leq \tfrac{\beta^2}{2\la_2} + \beta\sqrt{S} + \la_2 S.
\end{align}
Then 
$|\hOT - \OT| \leq \tfrac{\beta^2}{2\la_2} + \beta\sqrt{S} + \la_2 S.$
To conclude, by noting that $S \leq \tfrac{\la_1}{\la_2}\tr(A_\star) + (R(u_\star,v_\star) + \tfrac{\beta}{2\la_2^2})^2$, since $\gamma \leq \beta$, and that $(a + b + c)^2 \leq 3 a^2 + 3b^2 + 3c^2$ for any $a,b,c \geq 0$, we have
\begin{align}
    \tfrac{\beta^2}{2\la_2} + \beta\sqrt{S} + \la_2 S &\leq 2 \la_2 (\tfrac{\beta}{2\la} + \sqrt{S})^2 \leq 2 \la_2 \left(R(u_\star,v_\star) + \tfrac{\beta}{\la_2} + \sqrt{\tfrac{\la_1}{\la_2}\tr(A_\star)}\right)^2 \\
    & \leq 6 \la_2 (R^2(u_\star,v_\star) + \tfrac{\beta^2}{\la_2^2} + \tfrac{\la_1}{\la_2}\tr(A_\star)).
\end{align} 
\end{proof}

\section{Additional results on convex sets and random points}\label{sect:add_convex}

We first recall that the following property about bounded sets in Euclidean spaces.

\begin{lemma}[\cite{krieg2020random}]\label{lm:uic-convex-set}
Let $Z \subset \R^q$, $q \in \N$ be a non-empty open set. The following holds:
\begin{enumerate}
    \item If $Z$ is a {\em bounded Lipschitz domain} \citep{grisvard2011elliptic}, then it satisfies the {\em uniform interior cone condition} \citep{wendland2004scattered}.
    \item If $Z$ is a convex bounded set, then it is a bounded Lipschitz domain.
\end{enumerate}
\end{lemma}
\begin{proof}
For the first point, see Lemma 5 of \cite{krieg2020random} or Theorem 1.2.2.2 of \cite{grisvard2011elliptic}. For the second point, see Lemma 4 of \cite{krieg2020random} or Lemma 7 in \cite{dekel2004whitney}. Also, the fact that a convex bounded open set has the uniform interior cone condition is implied by Proposition 11.26 of \cite{wendland2004scattered}, that proves it for the more general class of sets called {\em star shaped sets w.r.t.\ a ball}.
\end{proof}

\begin{lemma}[Fill distance of i.i.d points on a u.i.c. set \citep{reznikov2016covering}]\label{lm:uic-sampling}
Let $Z \subset \R^q$, $q \in \N$ be a non-empty open set satisfying the uniform interior cone condition (see \cref{lm:uic-convex-set}). Let $\tilde{Z}_\ell$ be a collection of $\ell$ points sampled independently and uniformly at random from a probability $\rho$ that admits density (denote it by $p$) and such that $p(z) \geq c_0 > 0$ for any $z \in Z$. Let $\delta \in (0,1]$. Then there exist $\ell_0, C_1, C_2$ depending on $c_0, Z, \rho, q, r_0$ such that for $\ell \geq \ell_0$, the following holds with probability at least $1 - \delta$:
$$h_\ell \leq (C_1 \ell^{-1} \log(C_2 \ell/\delta))^{1/q}.$$
\end{lemma}
\begin{proof}
The uniform interior cone condition guarantees that there exists a cone $C$ such that for any $z \in Z$ there exists a spherical cone of radius $r$ such that $C_z \subseteq Z$ that is congruent to $C$ and with vertex in $z$. Then, for any $r \leq r_0$ there exists $c_1$ such that
$$\rho(B(z,r) \cap Z) \geq \rho(B(z,r) \cap C_z) \geq c_0 \frac{\textrm{vol}(B(z,r) \cap C_z)}{\textrm{vol}(Z)} \geq \frac{c_0 c_1}{\textrm{vol}(Z)} r^q.$$
Then we can apply Theorem 2.1 of \cite{reznikov2016covering} with $\Phi(r) = \frac{c_0 c_1}{\textrm{vol}(Z)} r^q$ obtaining that there exists $\ell_0, C_1, C_2$ depending on $c_0, c_1, q, r_0, Z$ such that with probability at least $1 - \delta$,
$$h_\ell \leq (C_1 \ell^{-1} \log(C_2 \ell /\delta))^{1/q}.$$
\end{proof}

\section{Representer theorem and coercivity for Problem~\eqref{eq:w_hat}}\label{sec:rep_thm}

We first adapt the proof of the representer theorem from \cite{marteau2020non}. We use it to prove coercivity of the functional.
Given the RKHS $\hhx, \hhy, \hhxy$ and with the same notation of Problem~\eqref{eq:w_hat} define  $\hat{\hh}_X = \textrm{span}\{\hat{w}_\mu, \phi_X(\tilde{x}_1),\dots, \phi_X(\tilde{x}_\ell)\}$, $\hat{\hh}_Y = \textrm{span}\{\hat{w}_\nu, \phi_Y(\tilde{y}_1),\dots, \phi_Y(\tilde{y}_\ell)\}$ and $\hat{\hh}_{XY} = \textrm{span}\{\phi(x_1,y_1),\dots,\phi(x_\ell,y_\ell)\}$.

\begin{lemma}[Representer theorem, \cite{marteau2020non}]\label{lm:representer}
Let $\la_1, \la_2 > 0$. Denote by $V_{\la_1, \la_2}(u, v, A)$ the objective function of Problem~\eqref{eq:w_hat}.
Let $u_1 \in \hat{\hh}_X , u_2 \in \hat{\hh}_X^\bot$, $v_1 \in \hat{\hh}_Y, v_2 \in \hat{\hh}_Y^\bot$ and $A_1 \in \hat{\hh}_{XY} \otimes \hat{\hh}_{XY},  A_2 \in \hat{\hh}_{XY}^\bot \otimes \hat{\hh}_{XY}, A_3 \in \hat{\hh}_{XY}^\bot \otimes \hat{\hh}_{XY}^\bot$. Assume that $u_2$ or $v_2$ or $A_2$ or $A_3$ are different from $0$. Set $u = u_1 + u_2$, $v = v_1 + v_2$ and $A = A_1 + A_2 + A_2^* + A_3$ and assume that $u, v, A$ is an admissible point for Problem~\eqref{eq:w_hat} Then
\begin{enumerate}
    \item $(u_1, v_1, A_1)$ is an admissible point for Problem~\eqref{eq:w_hat},
    \item $V_{\la_1, \la_2}(u_1, v_1, A_1) > V_{\la_1, \la_2}(u, v, A)$.
\end{enumerate}
\end{lemma}
\begin{proof}
We can decompose any $u \in \hhx, v \in \hhy, A \in \pdm{\hhxy}$ as $u_1 \in \hat{\hh}_X , u_2 \in \hat{\hh}_X^\bot$, $v_1 \in \hat{\hh}_Y, v_2 \in \hat{\hh}_Y^\bot$ and $A_1 \in \hat{\hh}_{XY} \otimes \hat{\hh}_{XY},  A_2 \in \hat{\hh}_{XY}^\bot \otimes \hat{\hh}_{XY}, A_3 \in \hat{\hh}_{XY}^\bot \otimes \hat{\hh}_{XY}^\bot$.
Note that the components $u_2, v_2, A_2, A_3$ do not impact the constraints or the functional but are only penalized by the regularizer, indeed $\scal{u_2}{\phi_X(\tilde{x}_i)}_\hhx = 0$ since $\phi_X(\tilde{x}_i) \in \hat{\hh}_X$, while $u_2 \in \hat{\hh}_X^\bot$, then
$$u(\tilde{x_i}) = \scal{u}{\phi_X(\tilde{x}_i)}_\hhx = \scal{u_1}{\phi_X(\tilde{x}_i)}_\hhx + \scal{u_2}{\phi_X(\tilde{x}_i)}_\hhx = \scal{u_1}{\phi_X(\tilde{x}_i)}_\hhx = u_1(\tilde{x}_i).$$
The same reasoning holds for $v(\tilde{x}_i)$ and for $\scal{u}{w_\mu}_\hhx, \scal{v}{w_\nu}_\hhy$. For $A$ analogously we have that $t = A_2\phi(\tilde{x}_i,\tilde{y}_i) \in \hat{\hh}_{XY}^\bot$ so
$\scal{\phi(\tilde{x}_i,\tilde{y}_i)}{A_2\phi(\tilde{x}_i,\tilde{y}_i)}_\hhxy = \scal{\phi(\tilde{x}_i,\tilde{y}_i)}{t}_\hhxy = 0$ and similarly for $A_3$, we have $\scal{\phi(\tilde{x}_i,\tilde{y}_i)}{A_3\phi(\tilde{x}_i,\tilde{y}_i)}_\hhxy = 0$. 
Let's see what happens to the penalization terms. For the quadratic term, we have $-\|u\|^2_\hhx = -\|u_1\|^2_\hhx - \|u_2\|^2_\hhx < -\|u_1\|^2_\hhx$ and analogously for $v$. For the trace term we have $\tr(A) = \tr(A_1) + \tr(A_3)$. Moreover $A \in \pdm{\hhxy}$ implies that $A_1 \succeq 0$ and by the Schur complement property  $A_3 \succeq A_2^* A_1^{-1} A_2 \succeq 0$. Then if $A_3 \succeq 0$ and different from zero we have $-\tr(A) < -\tr(A_1)$. If $A_2$ is different from zero this implies by the Schur complement property that $A_3$ is a positive definite operator different from zero and so $-\tr(A) < -\tr(A_1)$.
\end{proof}

\begin{lemma}[Problem~\eqref{eq:w_hat} has a maximizer]\label{lm:prob_w_hat_has_minimizer}
When $\la_1, \la_2 > 0$ and when there exists an admissible point $(\bar{u},\bar{v},\bar{A})$ with $\bar{u} \in \hhx, \bar{v} \in \hhy, \bar{A} \in \pdm{\hhxy}$, Problem~\eqref{eq:w_hat} admits a maximizer.
\end{lemma}
\begin{proof}
First define $S = \hhx \times \hhy \times (\hhxy \otimes \hhxy)$ and $\hat{S} =  \hat{\hh}_X \times \hat{\hh}_Y \times (\hat{\hh}_{XY} \otimes \hat{\hh}_{XY})$.
From the lemma above, we have that for any admissible point in $S \setminus \hat{S}$ there exists another admissible point in $\hat{S}$ that has a strictly larger value. Note moreover that $\hat{S}$ is a finite dimensional Hilbert space (with dimension at most $\ell^2(\ell + 1)^2$) and that (minus) the functional of Problem~\eqref{eq:w_hat} is coercive on it. Note moreover that \cref{eq:w_hat} is bounded from above since $V_{\la_1,\la_2}(u,v,A) \leq \|u\|_\hhx \|\hat{w}_\mu\|_\hhx + \|v\|_\hhy \|\hat{w}_\nu\|_\hhy - \la_2(\|u\|^2_\hhx + \|v\|^2_\hhy)$ has a maximum in $u,v$. Since Problem~\eqref{eq:w_hat} has also an admissible point by assumption, then it admits at least one maximizer \cite[see, e.g., Proposition 3.2.1, Page 119 of][]{bertsekas2009convex}.
\end{proof}

\section{Proof of Corollary~\ref{cor:costs} and Theorem~\ref{thm:complexities}}\label{sect:kme}

\begin{proof}{\bf of Corollary~\ref{cor:costs}.} 
The result is obtained by plugging in \cref{thm:uniform} the bounds on $\gamma = \|w_\mu - \hat{w}_\mu\|_\hhx + \|w_\nu - \hat{w}_\nu\|_\hhy$.
We deal now with the three scenarios.
\paragraph{(Exact integral)} In this case, since $\hat{w}_\mu := w_\mu, \hat{w}_\nu := w_\nu$, then $\gamma = 0$.

\paragraph{(Evaluation)} We use here the construction used to analyze kernel least squares, e.g., from \cite{caponnetto2007optimal,rosasco2010learning} \citep[see also][]{steinwart2008support}. We will do the construction for $w_\mu$, since the case of $w_\nu$ is analogous. We recall that $\phi_X: X \to \hhx$ is uniformly bounded and continuous on $X$, since we are considering the Sobolev kernel (see Proposition~\ref{ex:sobolev-kernel}).  Denote by $T \in \pdm{\hhx}$ the operator defined as $T = \int \phi_X(x) \otimes \phi_X(x) dx$ where $dx$ is the Lebesgue measure. Note that $T$ is trace class, indeed $\tr(T) = \int \tr(\phi_X(x) \otimes \phi_X(x)) dx = \int \|\phi_X(x)\|^2_\hhx dx \leq \textrm{vol}(X) \sup_{x \in X} \|\phi_X(x)\|^2_\hhx < \infty$. For any $f, g \in \hhx$, by the reproducing property
\begin{align}
    \scal{f}{Tg}_{\hhx} &= \int \scal{f}{(\phi_X(x) \otimes \phi_X(x)) g}_\hhx dx \\
    & = \int \scal{f}{\phi_X(x)}_\hhx\scal{g}{\phi_X(x)}_\hhx dx = \int f(x) g(x) dx.
\end{align}
In particular the equation above implies that $\|f\|^2_{L^2(X)} = \|T^{1/2} f\|^2_\hhx$ for any $f \in \hhx$.
Now, note that by assumption in this Corollary, we have chosen the kernel $k_X = k_{m+1}$ so $\hhx$ corresponds to the Sobolev space $\hhx = H^{m+1}(X)$ (see Proposition~\ref{ex:sobolev-kernel}). Now by \cref{assum:measures}, $\mu$ has a density that we denote $g_\mu$, that is differentiable up to order $m$ and such that all the derivatives of order $m$ are Lipschitz continuous. Since $X$ is convex and bounded, then Lipschitz continuity of all the derivatives of order $m$ implies that all the weak derivatives up to order $m+1$ are in $L^\infty(X)$. Since $L^\infty(X) \subset L^2(X)$, by the boundedness of $X$, we have that all the weak derivatives of $g_\mu$ up to order $m+1$ belong to $L^2(X)$, i.e. $\|g_\mu\|_{H^{m+1}(X)} < \infty$ so $g_\mu \in H^{m+1}(X) = \hhx$. So, by the reproducing property
\begin{align}
    w_\mu & = \int \phi_X(x) d\mu(x) = \int \phi_X(x) g_\mu(x) dx \\
    &= \int \phi_X(x) \scal{\phi_X(x)}{g_\mu}_\hhx dx = \int (\phi_X(x) \otimes \phi_X(x)) g_\mu \, dx = T g_\mu.
\end{align}
Now, the estimator $\hat{w}_\mu$ is defined as $\hat{w}_\mu = \int_X \phi_X(x) \hat{g}_\mu(x),$ where $\hat{g}_\mu \in \hhx$ is the {\em Kernel Least Squares estimator} \citep[see][]{narcowich2005sobolev,caponnetto2007optimal} of the density of $\mu$ that is $g_\mu$. With the same reasoning as above, we see that $\hat{w}_\mu = T \hat{g}_\mu$.
Now note that 
$$\|w_\mu - \hat{w}_\mu\|_\hhx = \|T(g_\mu - \hat{g}_\mu)\|_\hhx \leq \|T^{1/2}\|_{op} \|T^{1/2}(g_\mu - \hat{g}_\mu)\|_\hhx = \|T^{1/2}\|_{op} \|g_\mu - \hat{g}_\mu\|_{L^2(X)}.$$
Now $\|T^{1/2}\|^2_{op} = \|T^{1/2}\|^2_{op} \leq \tr(T) < \infty$ as we have seen above. Moreover $\|g_\mu - \hat{g}_\mu\|_{L^2(X)}$ is controlled by classical results on approximation theory, e.g.\ Proposition 3.2 of \cite{narcowich2005sobolev} (applied with $\alpha = 0$ and $q = 2$). It is possible to apply such results as the set $X$ is convex bounded and so it satisfies the required {\em uniform interior cone condition} \citep{wendland2004scattered} (see Lemma~\ref{lm:uic-convex-set}, \cpageref{lm:uic-convex-set}). The result guarantees that there exists two constants $C, h_0$ depending only on $X$, such that $\|g_\mu - \hat{g}_\mu\|_{L^2(X)} \leq C h^{m+1} \|g_\mu\|_\hhx$, where $h$ is the {\em fill distance} (see \cref{eq:fill-distance}) of the sampled $n_\mu$ points, with respect to $X$. Now by Lemma~\ref{lm:uic-sampling},  \cpageref{lm:uic-sampling} we have that $h \leq (C' n_\mu \log(C'' n_\mu/\delta))^{1/d}$ with probability at least $1-\delta$ for some constants $C', C''$ depending on $d, X$. Then, finally for some constant $C'''$ we have
\begin{align}
    \|w_\mu - \hat{w}_\mu\|_\hhx &\leq \tr(T)^{1/2} \|g_\mu - \hat{g}_\mu\|_{L^2(X)}\\
    & \leq \tr(T)^{1/2} C \|g_\mu\|_\hhx (C' n_\mu \log(C'' n_\mu/\delta))^{1/d} \leq C''' n_\mu^{-(m+1)/d} \log(n_\mu/\delta).
\end{align} 

\paragraph{(Sampling)} In this case we have $\hat{w}_\mu = \frac{1}{n_\mu} \sum_{i \in [n_\mu]} \phi_X(x_i)$ where $x_1, \dots, x_n$ are independently and identically distributed according to $\mu$. Then $\xi_i = \phi_X(x_i)$ for $i \in [n_\nu]$ are i.i.d. random vectors and $\hat{w}_\mu = \mathbb{E} \xi_1$. Since $\phi_X:X \to \hhx$ is uniformly bounded on $X$ (see Proposition~\ref{ex:sobolev-kernel}) denote by $c$ that bound. We have that $\|\xi_i\|_\hhx \leq c$ almost surely and $\mathbb{E} \|\xi_i - \mathbb{E} \xi_i\|^2_\hhx \leq 2c$ for all $i \in [n_\mu]$. Then we can apply the Pinelis inequality \citep[see Proposition 2 of][]{caponnetto2007optimal}, to control $\hat{w}_\mu = \frac{1}{n_\mu} \sum_{i=1}^{n_\mu} \xi_i$, for which the following holds with probability $1-\delta$
$$\|w_\mu - \hat{w}_\mu\|_\hhx = \|\mathbb{E} \xi_1 - \frac{1}{n_\mu} \sum_{i=1}^{n_\mu} \xi_i\|_\hhx \leq 4c n_\mu^{-1/2} \log \frac{6 n_\mu}{\delta}.$$
\end{proof}

\begin{proof}{\bf of Theorem~\ref{thm:complexities}.}
This theorem is a direct consequence of Corollary~\ref{cor:costs}. Let $\eps > 0, \ell, n_\mu, n_\nu \in \N$. We denote by $f(x) \asymp g(x)$ the fact that there exists two constants $0 < C_1 \leq C_2$ not depending on $x$ such that $C_1 g(x) \leq f \leq C_2 g(x)$ (in our particular case this means that the constants will not depend on $\eps, \ell, n_\mu, n_\nu$).
For the rest of the proof we will fix $\ell \asymp \eps^{-2d/(m-d)}$. Indeed, this choice implies that $\ell^{-(m-d)/2d} = O(\eps)$. We recall, moreover, from \cref{eq:cost} (that is proven in \cref{sec:algo}) that the computational time to achieve the estimator is $\tilde{O}(C + E \ell + \ell^{3.5})$ and the required memory is $O(\ell^2)$, where $E, C$ are specified for each scenario in \cref{sec:mean_estimators}.
Now we will quantify the complexity for the three scenarios.

\paragraph{(Exact integral)} In this case, by Corollary~\ref{cor:costs}
$$|\hOT - \OT| = \tilde{O}(\ell^{-(m-d)/2d}) = \tilde{O}(\eps).$$
Note that, since for this scenario $C = O(1), E = O(1)$, the complexity in time is 
$$\tilde{O}(C + E\ell + \ell^{3.5}) = \tilde{O}(\ell^{3.5}) = \tilde{O}(\eps^{-7d/(m-d)}).$$
In space, analogously $O(\ell^2) = O(\eps^{-4d/(m-d)})$.

\paragraph{(Evaluation)} In this case we choose $n_\nu \asymp n_\mu \asymp  O(\eps^{-d/(m+1)})$. Indeed, with this choice $n_\mu^{-(m+1)/d} = O(\eps)$ (analogously for $n_\nu$). Then, by Corollary~\ref{cor:costs}
$$|\hOT - \OT| = \tilde{O}(n_\mu^{-(m+1)/d} + n_\nu^{-(m+1)/d} + \ell^{-(m-d)/2d}) = \tilde{O}(\eps).$$
From a computational viewpoint $C = O(n_\mu^3 + n_\nu^3)$ and $E = O(n_\mu + n_\nu)$. Then the time complexity is
\begin{align}
    \tilde{O}(C + E\ell + \ell^{3.5}) &= \tilde{O}(n_\mu^3 + n_\nu^3 + (n_\mu + n_\nu)\ell + \ell^{3.5}) \\
    &= \tilde{O}(\eps^{-3d/(m+1)} + \eps^{-d/(m+1)-2d/(m-d)} + \eps^{-7d/(m-d)}) = O(\eps^{-7d/(m-d)}),
\end{align}
wile the space complexity is $O(\ell^2) = O(\eps^{-4d/(m-d)})$.

\paragraph{(Sampling)} In this case we choose $n_\nu \asymp n_\mu \asymp  O(\eps^{-2})$. Indeed with this choice $n_\mu^{-1/2} = O(\eps)$ (analogously for $n_\nu$). Then, by Corollary~\ref{cor:costs}
$$|\hOT - \OT| = \tilde{O}(n_\mu^{-1/2} + n_\nu^{-1/2} + \ell^{-(m-d)/2d}) = \tilde{O}(\eps).$$
From a computational viewpoint $C = O(n_\mu^2 + n_\nu^2)$ and $E = O(n_\mu + n_\nu)$. Then the time complexity is
\begin{align}
\tilde{O}(C + E\ell + \ell^{3.5}) &= \tilde{O}(n_\mu^2 + n_\nu^2 + (n_\mu + n_\nu)\ell + \ell^{3.5}) \\
&= \tilde{O}(\eps^{-4} + \eps^{-2 - 2d/(m-d)} + \eps^{-7d/(m-d)}) = \tilde{O}(\eps^{-\max(4,7d/(m-d))}).
\end{align}
Also in this case the space complexity is $O(\ell^2) = O(\eps^{-4d/(m-d)})$.
\end{proof}





\section{Dual Algorithm and Computational Bounds}\label{sec:algo}

In this section, we describe a dual algorithmic procedure to compute \cref{eq:ot_hat_primal}, and bound the computational complexity and memory footprint it requires to achieve a given precision. We start by deriving \cref{eq:w_hat_dual}, the dual formulation of \cref{eq:w_hat}, and express \cref{eq:ot_hat_primal} as a function of the dual solution $\hat{\gamma}$.

\begin{theorem}\label{thm:dual_problem}
    The dual problem of \cref{eq:w_hat} is \cref{eq:w_hat_dual}. Further, the estimator $\hOT$ can be expressed as a function of the solution $\hat{\gamma}$ of \cref{eq:w_hat_dual}:
    \begin{equation}
    \hOT = \frac{q^2}{2\la_2} - \frac{1}{2\la_2}\sum_{j=1}^\ell\hat{\gamma}_j (\hat{w}_\mu(\tx_j) + \hat{w}_\nu(\ty_j)).
    \end{equation}
\end{theorem}

\begin{proof}{\bf of \Cref{thm:dual_problem}.}
\paragraph{Finite-dimensional representation.}
\BM{Added this paragraph, to introduce $\bB$, etc.}
Let us start by formulating \eqref{eq:w_hat} using a finite-dimensional representation for $A\in\pdm{\hhxy}$. Following \citet{marteau2020non}, observe that problem \eqref{eq:w_hat} needs only to be solved w.r.t. $A$ in the finite-dimensional Hilbert space spanned by $\{\phi(\tx_j, \ty_j) : j \in [\ell]\}$. This result is formally proven for our setting in \Cref{sec:rep_thm}. Therefore, it is sufficient to consider positive operators of the form $A = \sum_{ij = 1}^{\ell} C_{ij} \phi(\tx_i, \ty_i)\otimes\phi(\tx_j, \ty_j)$, with $\bC \in \mathbb{S}_+(\RR^\ell)$. With this parameterization, we have $\tr A = \tr(\bC\bK)$ and 
$\dotp{\phi(\tx_i, \ty_i)}{\A\phi(\tx_j, \ty_j)} = (\bK\bC\bK)_{ij},$ where $\bK = [k_{XY}((\tx_i, \ty_i), (\tx_j, \ty_j)]_{ij=1}^\ell$. 
As \cite{rudi2020global}, we can thus consider the Cholesky decomposition $\bK = \bR\bR^\top$ (or alternatively take the square root $\bR = \bK\srt$), and represent problem \eqref{eq:w_hat} in terms of the columns $\{\Phi_{j} : j \in [\ell]\}$ of $\bR$ and solve directly for $\bB = \bR\bC\bR^\top$:
\begin{equation}\label{eq:w_hat_finite_dim}
\begin{aligned}
    &\underset{\substack{ u \in \hhx, v \in \hhy  \\ \bB \in \SS_+(\RR^\ell) }}{\max} && \dotp{u}{\hat{w}_\mu} + \dotp{v}{\hat{w}_\nu} - \lambda_1\tr \bB - \lambda_2(\|u\|_{\Hcal_X}^2 + \|v\|_{\Hcal_Y}^2) \\
     & \st \forall j\in [\ell], 
     && c(\tx_j, \ty_j) - u(\tx_j) - v(\ty_j) = \Phi_j^\top\bB\Phi_j.
 \end{aligned}
\end{equation}
\paragraph{Deriving the dual.}
Next, let us observe that problem \eqref{eq:w_hat} is convex, and admits a feasible point by Corollary~\ref{cor:asm1-Astar} and a maximizer by Lemma \ref{lm:prob_w_hat_has_minimizer}. The same applies to \eqref{eq:w_hat_finite_dim}. Therefore, strong duality holds. 
The Lagragian of \cref{eq:w_hat_finite_dim} is
\begin{align}\label{eq:lagragian}
    \begin{split}
    \Lcal(u, v, \bB, \gamma) &=  \dotp{u}{\hat{w}_\mu} + \dotp{v}{\hat{w}_\nu} - \lambda_1\tr \bB - \la_2\|u\|_\hhx^2 - \la_2\|v \|_\hhy^2 \\
    &+ \sum_{i=1}^\ell\gamma_{i}(c(\tilde{x}_i, \tilde{y}_i) - \dotp{u}{\phi_X(\tilde{x}_i)} - \dotp{v}{\phi_Y(\tilde{y}_i)} - \Phi_i^\top\bB\Phi_i).
    \end{split}
\end{align}
At the optimum, we have $\nabla_u \Lcal(u, v, \bB, \gamma) = 0$ and $\nabla_v \Lcal(u, v, \bB, \gamma) = 0$ , which yields
\begin{align}\label{eq:uv_opt}
    \begin{split}
    u &= \frac{1}{2\lambda_2}(\hat{w}_\mu - \sum_{i=1}^\ell \gamma_i \phi_X(\tilde{x}_i)) \\
    v &= \frac{1}{2\lambda_2}(\hat{w}_\nu - \sum_{i=1}^\ell \gamma_i \phi_Y(\tilde{y}_i)).
    \end{split}
\end{align}
Let us now derive the optimality condition on $\bB$: we have
\begin{align}\label{eq:B_opt}
\begin{split}
    \underset{\bB\in\SS_+(\RR^\ell)}{\sup} -\sum_{i=1}^\ell\gamma_{i}\Phi_i^\top\bB\Phi_i - \lambda_1\tr \bB
    &= \underset{\bB\in\SS_+(\RR^\ell)}{\sup} \dotp{\bB}{ -(\sum_{i=1}^\ell\gamma_{i}\Phi_i\Phi_i^\top + \lambda_1 \eye_\ell)}\\
    & = \begin{cases} 0 & \mbox{if } \sum_{i=1}^\ell\gamma_{i}\Phi_i^\top\Phi_i + \lambda_1 \eye_\ell \succeq 0 \\
    -\infty & \mbox{otherwise}. 
    \end{cases}
\end{split}
\end{align}
Plugging \cref{eq:uv_opt} and \cref{eq:B_opt} in the \cref{eq:lagragian}, we get \eqref{eq:w_hat_dual}. Finally, using \cref{eq:uv_opt}, we have
\begin{align}
    \hOT &= \dotp{\hat{u}}{\hat{w}_\mu}_\hhx + \dotp{\hat{u}}{\hat{w}_\nu}_\hhx = \frac{q^2}{2\la_2} - \frac{1}{2\la_2}\sum_{j=1}^\ell\hat{\gamma}_j (\hat{w}_\mu(\tx_j) + \hat{w}_\nu(\ty_j)).
\end{align}
\end{proof}

To solve \eqref{eq:w_hat_dual}, it is possible to use standard software packages~\citep{boyd2004convex}. Alternatively, it can be made more scalable by adding a self-concordant barrier term to \eqref{eq:w_hat_dual} and using interior point methods with Newton steps. For a given barrier penalization $\delta > 0$, we thus aim to solve
 \begin{align}\label{eq:w_hat_dual_barrier}
\begin{split}
      \underset{\gamma \in \RR^\ell}{\min}~& \frac{1}{4\lambda_2} \gamma^\top {\bf Q} \gamma - \frac{1}{2\lambda_2}\sum_{j=1}^\ell \gamma_j z_j + \frac{q^2}{4\la_2}  - \frac{\delta}{\ell} \log\det (\sum_{i=1}^\ell\gamma_{i}\Phi_{i}\Phi_{i}^\top + \lambda_1 \eye_{\ell})) \\
      &~~~ \mbox{ such that } ~~~  \sum_{j=1}^\ell \gamma_j \Phi_j\Phi_j^\top + \lambda_1 \eye_\ell \succeq 0.
 \end{split}
\end{align}
Starting from an initial value $\delta_0$, the barrier method~\citep{nemirovski2004interior} consists in iteratively solving \cref{eq:w_hat_dual_barrier} (using Newton iterations), and progressively decreasing $\delta$. In \Cref{thm:dual_algorithm,thm:estimator_precision}, we precisely analyze the complexity of the barrier method applied to \cref{eq:w_hat_dual}, and bound the number of operations required to obtain an estimator of $\OT$ with a desired accuracy.
\begin{theorem}\label{thm:dual_algorithm}
Using a dual interior point method, a solution of problem \eqref{eq:w_hat} with value precision $O(\tau)$ can be obtained in $O(C + E\ell + \ell^{3.5}\log(\frac{\ell}{\tau}))$ operations and $O(\ell^2)$ memory, where $E$ is the cost of querying $\hat{w}_\mu$ and $\hat{w}_\nu$, and $C$ is the cost of computing $q^2$.
\end{theorem}

\begin{proof}{\bf of \Cref{thm:dual_algorithm}.}
Removing terms that are constant in $\gamma$, problem \eqref{eq:w_hat_dual_barrier} is equivalent to minimizing the dual functional
%
\begin{align}
\begin{split}
    J(\gamma) \defeq &\frac{1}{4\lambda_2}\gamma^\top\bQ \gamma - \frac{1}{2\lambda_2}\sum_{j=1}^\ell \gamma_j z_j -\frac{\delta}{\ell} \log\det (\sum_{i=1}^\ell\gamma_{i}\Phi_{i}\Phi_{i}^\top + \lambda_1 \eye_{\ell}).
\end{split}
\end{align}
Its gradient is 
\begin{align}
    \begin{split}
    J'(\gamma)_i = \frac{1}{2\lambda_2} (\bQ \gamma)_i - \frac{1}{2\lambda_2}z_i - \frac{\delta}{\ell} \Phi_i^\top(\Phi\diag(\gamma) \Phi^\top + \lambda_1 \eye_\ell)^{-1}\Phi_i,~~ i \in [\ell],
    \end{split}
\end{align}
and its Hessian
\begin{align}
    J''(\gamma)_{ij} = \frac{1}{2\lambda_2}\bQ_{ij} + \frac{\delta}{\ell} [\Phi_i^\top(\Phi\diag(\gamma) \Phi^\top + \lambda_1 \eye_\ell)^{-1}\Phi_j]^2,~~ i,j  \in [\ell].
\end{align}
From there, we may minimize $J(\gamma)$ using damped Newton iterations
\begin{align}\label{eq:damped_newton}
    \gamma' = \gamma - \frac{[J''(\gamma)]^{-1}J'(\gamma)}{1 + \sqrt{\frac{\ell}{\delta}}\lambda(\gamma)},
\end{align}
where $\lambda^2(\gamma) = J'(\gamma)^\top[J''(\gamma)]^{-1}J'(\gamma)$ is the Newton decrement, or using backtracking line-search Newton iterations~\citep{boyd2004convex}.

\paragraph{Number of iterations.}
$J''(\gamma)$ can be computed and inverted in $O(\ell^3)$ operations, and assuming $\hat{w}_\mu(\tx_i), i \in [l]$ and $\hat{w}_\nu(\ty_i), i \in [l]$ are precomputed, $J'(\gamma)$ can be computed in $O(\ell^3)$ operations, hence the complexity per iteration is $O(\ell^3)$. 

Let $F(\gamma) \defeq \frac{1}{4\lambda_2}\gamma^\top\bQ \gamma - \frac{1}{2\lambda_2}\sum_{j=1}^\ell \gamma_j z_j  + \frac{q^2}{4\la_2}$ be the objective function of \eqref{eq:w_hat_dual}. Since
$H(\gamma) \defeq - \log\det (\sum_{i=1}^\ell\gamma_{i}\Phi_{i}\Phi_{i}^\top + \lambda_1 \eye_{\ell}))$
is a self-concordant barrier function of concordance parameter $\ell$, standard results on barrier methods imply that $\delta$ controls the deviation (in value) to the optimum of $F$~\citep{nemirovski2004interior}. Moreover, a solution $\tilde{\gamma}$ to \eqref{eq:w_hat_dual} of precision $\tau > 0$, i.e. satisfying $F(\tilde{\gamma})- F(\hat{\gamma}) \leq \tau$ where $\hat{\gamma}$ is the optimum of \eqref{eq:w_hat_dual}, can computed in $O(\sqrt{\ell} \log\frac{\ell}{\tau})$ Newton iterations using an interior point method by progressively decreasing $\delta$ using a suitable scheme until $\delta \leq \tau$: see \citep{nemirovski2004interior}.

Hence, taking into account the $O(E \ell)$ cost of computing $z_j, j=1, ..., \ell$ and the $O(C)$ cost required to compute $q^2$, to achieve a precision $\tau > 0$ a total of $O(C +  E\ell +  \ell^{3.5}\log(\frac{\ell}{\tau}))$ operations and $O(\ell^2)$ memory are required. 

\end{proof}

In \Cref{thm:dual_algorithm}, we may use any of the kernel mean estimators presented in \Cref{sec:mean_estimators} and apply the corresponding computational costs $C$ and $E$. However, the given bounds only apply to the precision in value, i.e., on $ V_{\la_1, \la_2}(\hat{u}, \hat{v}, \hat{A}) - V_{\la_1, \la_2}(u, v, A)$, and not on the solutions $(u, v, A)$ themselves. In \Cref{thm:estimator_precision}, we derive bounds on the algorithmic approximation of the estimator in $\eqref{eq:hat_OT}$ obtained by minimizing \cref{eq:w_hat_dual_barrier} to a precision $\tau > 0$, as a function of its computational complexity.

\begin{theorem}\label{thm:estimator_precision}
Under the same notation and assumptions as \Cref{thm:sample_bound}, let $\tilde{\gamma}$ be obtained by running a barrier method on \eqref{eq:w_hat_dual} to precision $\tau > 0$, i.e.,~by iteratively solving \eqref{eq:w_hat_dual_barrier} and progressively decreasing $\delta$ until $\delta \leq \tau$, as described by \citet{nemirovski2004interior}. Define $\tOT$ as
$$\tOT = \frac{q^2}{2\la_2} - \frac{1}{2\la_2}\sum_{j=1}^\ell\tilde{\gamma}_j (\hat{w}_\mu(\tx_j) + \hat{w}_\nu(\ty_j)).$$
Then $\tOT$ satisfies the following bound
\begin{equation}
 |\tOT - \OT| ~\leq~6 \lambda_2(\|u_\star\|^2_\hhx + \|v_\star\|^2_\hhy) + 6\tfrac{\beta^2}{\la_2} + 6 \lambda_1 \tr A_\star + 6\la_2\tau.
\end{equation}
Further, the considered algorithm to compute $\tOT$ has a cost of $O(C + E\ell + \ell^{3.5}\log(\frac{\ell}{\tau}))$ in time and $O(\ell^2)$ in memory. 
\end{theorem}

\begin{proof}{\bf of \Cref{thm:estimator_precision}.}
Let $\tilde{\gamma}$ be obtained by running a barrier method on \cref{eq:w_hat_dual} to precision $\tau$, i.e. by solving \eqref{eq:w_hat_dual_barrier} and decreasing $\delta$ until $\delta\leq \tau$. Then, from properties of barrier methods~\citep{nesterov1994interior} we can associate to $\tilde{\gamma}$ the following primal feasible points, obtained by nullifying the Lagragian of \eqref{eq:w_hat_dual_barrier} at $\tilde{\gamma}$:
\begin{align}\label{eq:ab_opt}
    \begin{split}
    \tilde{u} &= \frac{1}{2\lambda_2}(\hat{w}_\mu - \sum_{i=1}^\ell \tilde{\gamma}_i \phi_X(\tilde{x}_i)) \\
    \tilde{v} &= \frac{1}{2\lambda_2}(\hat{w}_\nu - \sum_{i=1}^\ell \tilde{\gamma}_i \phi_Y(\tilde{y}_i)) \\
    \tilde{A} &= \sum_{ij=1}^\ell B_{ij} \phi(\tx_i, \ty_i)\otimes\phi(\tx_j, \ty_j),
    \end{split}
\end{align}
with $\bB = \frac{\delta}{\ell}(\sum_{i=1}^\ell\tilde{\gamma}_i\Phi_i\Phi_i^\top + \lambda_1\eye_\ell)^{-1}$. In particular, from properties of interior point methods~\citep{nemirovski2004interior,nesterov1994interior}, we have
\begin{enumerate}[(i)]
    \item $(\tilde{u}, \tilde{v}, \tilde{A})$ is a feasible point of \eqref{eq:w_hat},
    \item The duality gap between the objective $\hat{V}_{\lambda_1, \lambda_2}$ of \eqref{eq:w_hat} evaluated at $(\tilde{u}, \tilde{v}, \tilde{A})$ and the objective $F$ of \eqref{eq:w_hat_dual} at $\tilde{\gamma}$ is equal to $\delta$, i.e. $ F(\tilde{\gamma}) - \hat{V}_{\lambda_1, \lambda_2}(\tilde{u}, \tilde{v}, \tilde{A}) = \delta$.
\end{enumerate}

Further, note that we have 
\begin{align}
    \tOT &\defeq \frac{q^2}{2\la_2} - \frac{1}{2\la_2}\sum_{j=1}^\ell\tilde{\gamma}_j (\hat{w}_\mu(\tx_j) + \hat{w}_\nu(\ty_j))\\
    &= \dotp{\tilde{u}}{\hat{w}_\mu}_\hhx + \dotp{\tilde{v}}{\hat{w}_\nu}_\hhy.
\end{align}
Let us now bound $\tOT - \OT$. We will follow similar arguments than in the proof of \Cref{thm:sample_bound}, with an additional $\tau$ precision term.
As in the proof of \Cref{thm:sample_bound}, let $(\tilde{u}_\varepsilon, \tilde{v}_\varepsilon) = (\tilde{u} - \varepsilon/2, \tilde{v} - \varepsilon)$. Since $(\tilde{u}, \tilde{v}, \tilde{A})$ satisfy the constraints of \eqref{eq:w_hat}, from  the same arguments $(\tilde{u}_\varepsilon, \tilde{v}_\varepsilon)$ defines a feasible point for \eqref{EqDualOT}. Hence, we have the equivalent of \cref{eq:upperbound-tildeW-W}:
\begin{equation}\label{eq:upperbound-tildeW-W_2}
     \OT = V(u_\star, v_\star)  \geq V(\tilde{u}_\eps, \tilde{v}_\eps) = V(\tilde{u}, \tilde{v}) - \eps \geq \tOT - \Delta(\tilde{u}, \tilde{v}) - \eps.
\end{equation}
Next, let $\hat{\gamma}$ be the optimum of \eqref{eq:w_hat_dual}, and $F$ the objective function of $\eqref{eq:w_hat_dual}$. By strong duality (see \Cref{thm:dual_problem}), we have $F(\hat{\gamma}) = \hat{V}_{\la_1, \la_2} (\hat{u}, \hat{v}, \hat{A})$. Further, by optimality of $\hat{\gamma}$, we have $F(\hat{\gamma}) \leq F(\tilde{\gamma})$. Hence, using (ii) and $\delta \leq \tau$, we have
\begin{align}
\label{eq:bound-tildeW-to-Vastar_2}  
\hat{V}_{\la_1, \la_2}(\tilde{u}, \tilde{v}, \tilde{A}) 
& = F(\tilde{\gamma}) - \delta\\
& \geq F(\hat{\gamma}) - \delta\\
& = \hat{V}_{\la_1, \la_2}(\hat{u}, \hat{v}, \hat{A}) - \delta \\
& \geq \hat{V}_{\la_1,\la_2}(u_\star, v_\star, A_\star) - \tau\\
& = V(u_\star, v_\star) - \left[V(u_\star, v_\star) - \hat{V}_{\la_1,\la_2}(u_\star, v_\star, A_\star)\right] - \tau \\
& \geq \OT - \Delta(u_\star, v_\star) - \la_1\tr(A_\star) + \la_2 R^2(u_\star,v_\star) - \tau.
\end{align}

From there, the rest of the proof of \Cref{thm:sample_bound} follows identically: developing $\hat{V}_{\la_1, \la_2}(\tilde{u}, \tilde{v}, \tilde{A})$ and combining with \cref{eq:upperbound-tildeW-W_2}, we have
\begin{equation}\label{eq:bound-Wtilde-W_2}
    \la_1\tr(\tilde{A} - A_\star) +  \la_2 (R^2(\tilde{u}, \tilde{v}) - R^2(u_\star, v_\star)) - \Delta(u_\star, v_\star) - \tau \leq \tOT - \OT \leq \Delta(\tilde{u}, \tilde{v}) + \eps.
\end{equation} 
Hence
$$\la_1\tr(\tilde{A}) +  \la_2 R^2(\tilde{u}, \tilde{v}) \leq \la_1\tr(A_\star) +  \la_2 R^2(u_\star,v_\star) + \Delta(u_\star, v_\star) + \Delta(\tilde{u}, \tilde{v}) + \eps + \tau.$$
Therefore, replacing $S$ from the proof of \eqref{thm:sample_bound} with 
$S'  := S + \tau$, and $\hat{u}, \hat{v}, \hat{A}$ with $\tilde{u}, \tilde{v}, \tilde{A}$, the rest of the proof follows and eventually yields
$$
|\tOT - \OT| \leq 6 \lambda_2(R^2(u_\star, v_\star) + \frac{\beta^2}{\la_2^2} + \frac{\lambda_1}{\lambda_2} \tr A_\star + \tau).
$$

To conclude, note that as a consequence of \Cref{thm:dual_algorithm},
$\tOT$ can be computed in $O(C + E\ell + \ell^{3.5}\log(\frac{\ell}{\tau}))$ operations and $O(\ell^2)$ memory.

\end{proof}

\paragraph{Recovering unregularized optimal transport.}
In \cref{eq:w_hat} and in the case of empirical estimators $\hat{w}_\mu$ and $\hat{w}_\nu$ (see \Cref{sec:mean_estimators}), we considered the case where the sample pairs $(\tx_i, \ty_i), i \in [l]$ covering $X\times Y$ are distinct from the samples $x_i\sim \mu, i \in [n_\mu]$ and $y_j \sim \nu, j \in [n_\nu]$. However, covering $X\times Y$ with the $n_\mu n_\nu$ pairs given by the $\mu$ and $\nu$ samples $(x_i, y_j), i \in [n_\mu], j\in [n_\nu]$, we may rewrite \eqref{eq:w_hat_dual} as a regularized optimal transport problem:
%
\begin{equation}\label{eq:dual_mmd_reg}
\begin{aligned}
    &\underset{\Gamma \in \RR^{{n_\mu}\times {n_\nu}}}{\min} &&\sum_{ij} \Gamma_{ij} c(x_i, y_j) + \frac{1}{2\lambda_2} r^T\bK_X r + \frac{1}{2\lambda_2} c^T\bK_Y c\\
      &~~~ \st ~~~ && \sum_{\substack{i=1, ..., n_\mu\\ j=1,...,n_\nu}} \Gamma_{ij}\Phi_{ij}\Phi_{ij}^\top + \lambda_1 \eye_{{n_\mu}{n_\nu}}  \succcurlyeq 0,\\
      &&& r_i = \frac{1}{n_\mu} - \sum_{j=1}^{n_\nu}\Gamma_{ij},\quad i \in [n_\mu],\\
      &&& c_j = \frac{1}{n_\nu}- \sum_{i=1}^{n_\mu}\Gamma_{ij},\quad j \in [n_\nu],
\end{aligned}
\end{equation}
were we reindexed $\Phi_p, p \in [{n_\mu}{n_\nu}]$ as $\Phi_{ij}, i \in [{n_\mu}], j \in [{n_\nu}]$, and where $(\bK_X)_{ij} = k_X(x_i, x_j), i, j\in [{n_\mu}]$, $(\bK_Y)_{ij} = k_Y(y_i, y_j), i, j\in [{n_\nu}]$. Hence, \eqref{eq:dual_mmd_reg} can be interpreted as a regularized optimal transport problem, where $\Gamma \in \RR^{{n_\mu}\times {n_\nu}}$ plays the role of the transportation plan, and the marginal violations are penalized with maximum mean discrepancy (MMD) terms~\citep{gretton2012kernel}. When $\lambda_1$ goes to $0$, the $\mathrm{SDP}$ constraint becomes a $\Gamma \in \RR_+^{{n_\mu}\times {n_\nu}}$ positivity constraint, and when $\lambda_2$ goes to $0$, the MMD penalization terms enforce the hard constraints $\Gamma\11_{n_\nu}= \frac{\11_{n_\mu}}{{n_\mu}}$ and $\Gamma^\top\11_{n_\mu} = \frac{\11_{n_\nu}}{{n_\nu}}$. In particular, when $(\lambda_1, \lambda_2) \rightarrow (0, 0)$, we recover the unregularized OT problem between the empirical measures $\hat{\mu} = \frac{1}{{n_\mu}}\sum_{i=1}^{n_\mu}\delta_{x_i}$ and $\hat{\nu} = \frac{1}{{n_\nu}}\sum_{j=1}^{n_\nu}\delta_{y_j}$, which can be formally verified by deriving the dual of \eqref{eq:w_hat} with $\lambda_1$ and/or $\lambda_2$ equal to $0$.


\section{Numerical Experiments}\label{sec:numerical}

\begin{figure}[ht]
    \centering
    \includegraphics[width = .32\textwidth]{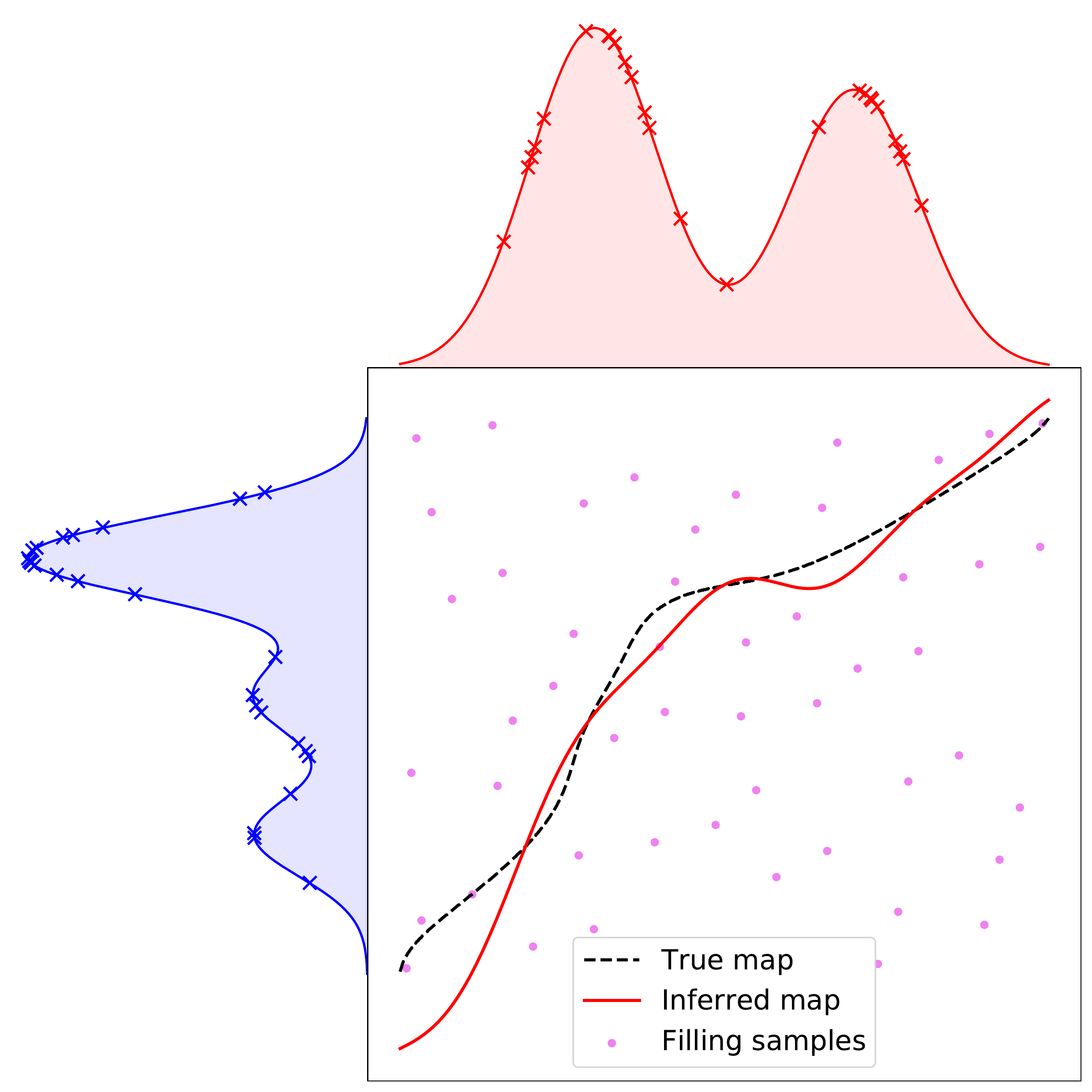}
    \includegraphics[width = .32\textwidth]{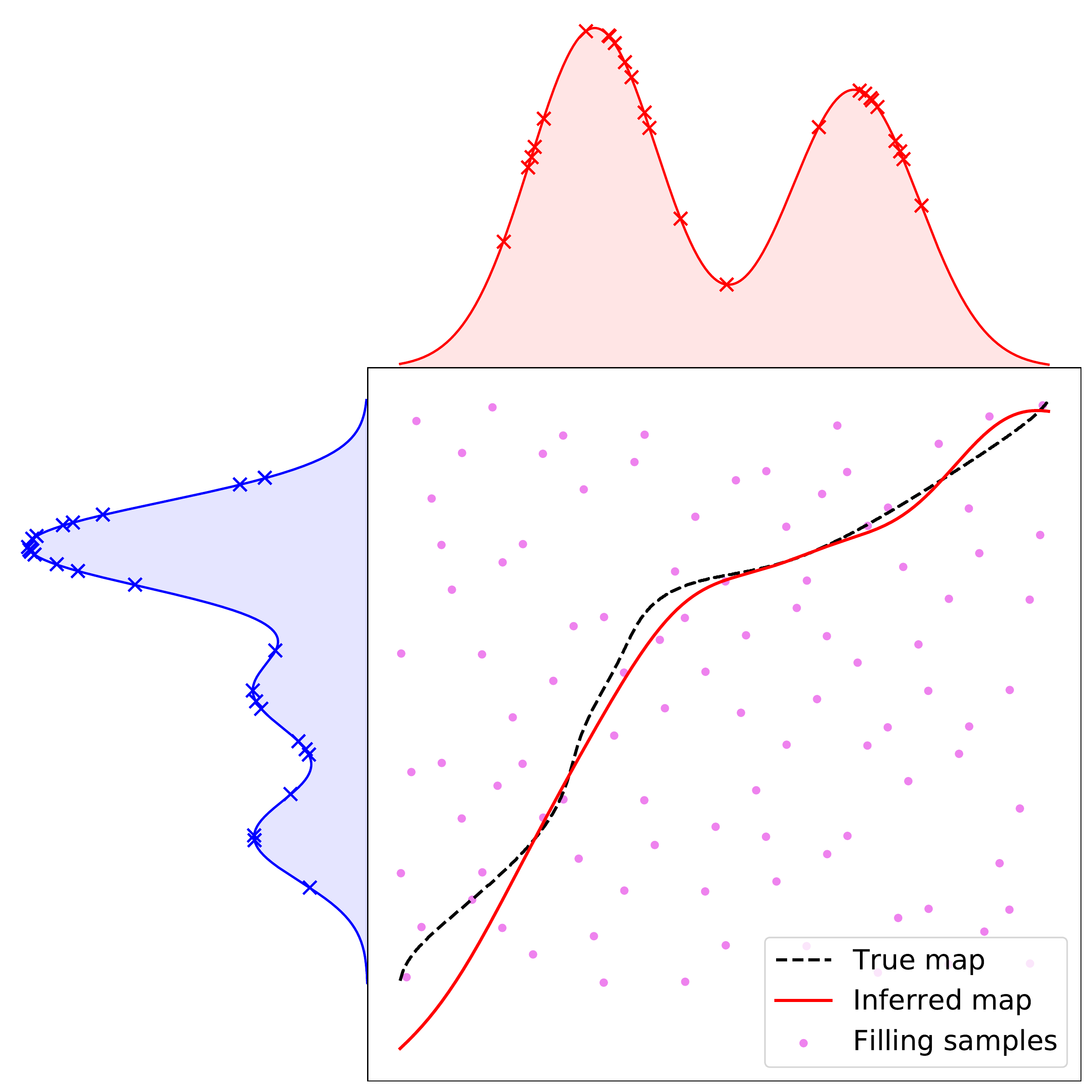}
    \includegraphics[width =
    .32\textwidth]{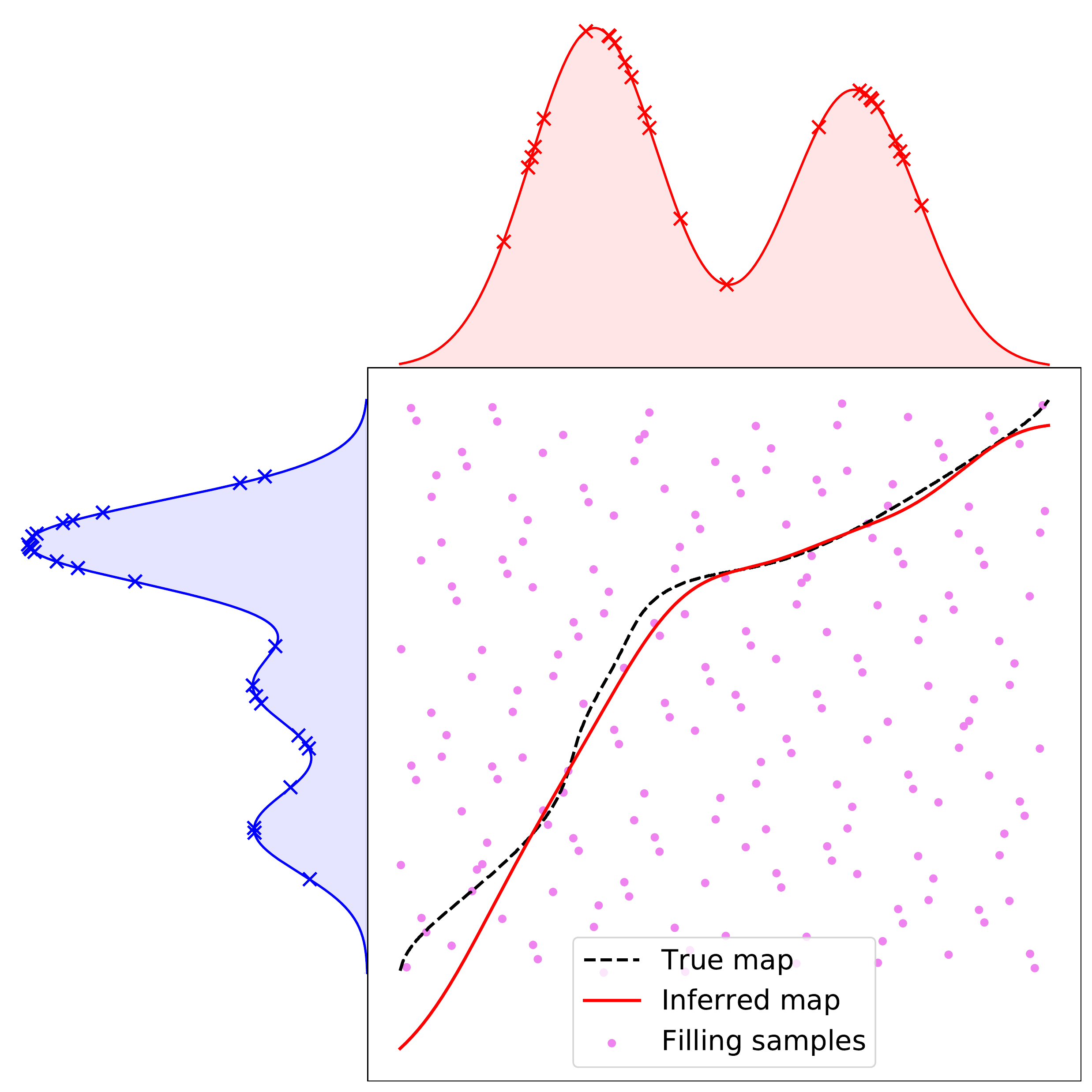}
    \caption{Effect of increasing the number of filling samples $\ell$ on the transportation map. \textit{(left)}: $\ell=50, {n_\mu} = {n_\nu} =25$, \textit{(middle)}: $\ell=100, n =25$, \textit{(right)}: $\ell=200, {n_\mu} = {n_\nu} =25$.}
    \label{fig:transport_map_fill}
\end{figure}

\begin{figure}[ht]
    \centering
    \includegraphics[width = .32\textwidth]{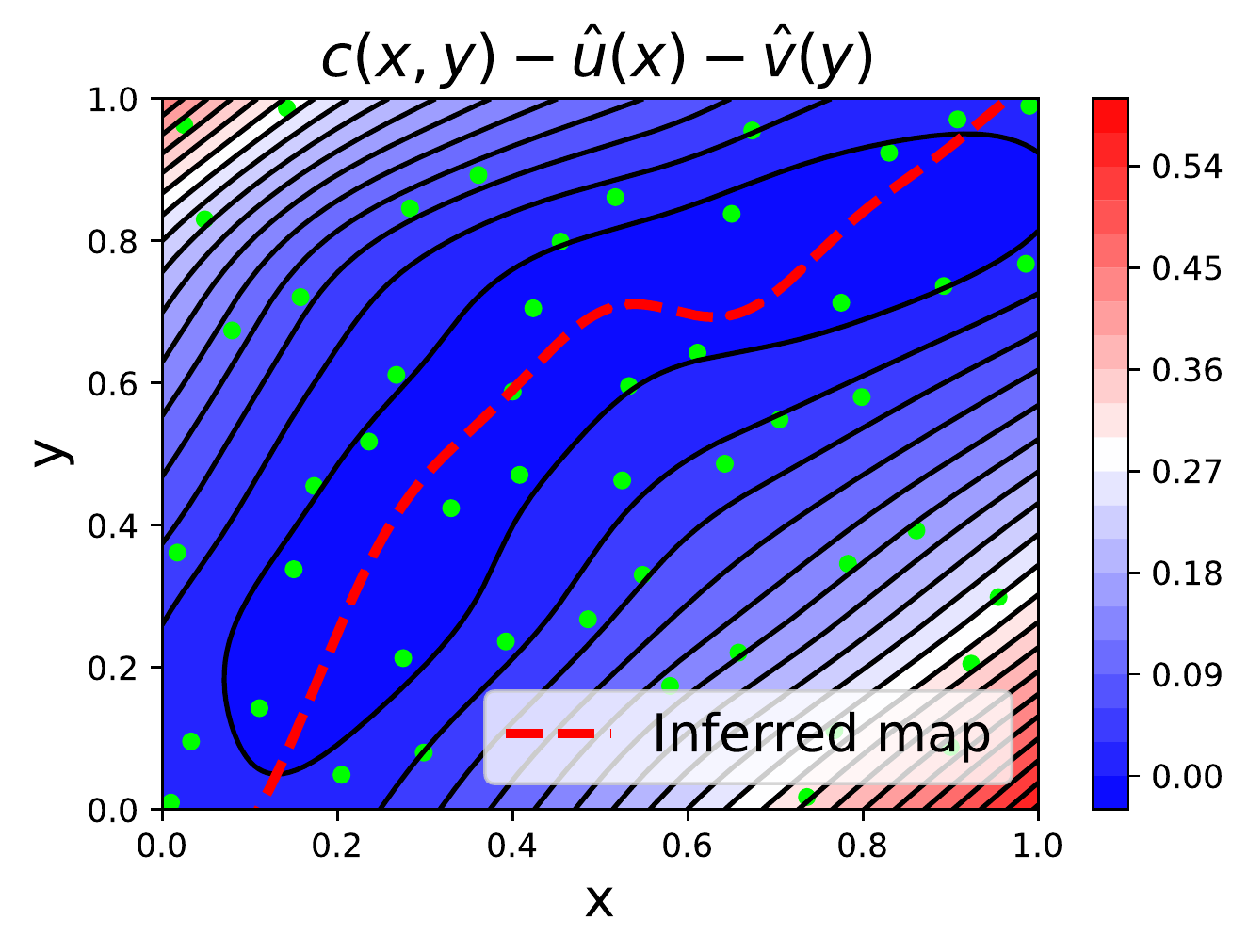}
    \includegraphics[width = .32\textwidth]{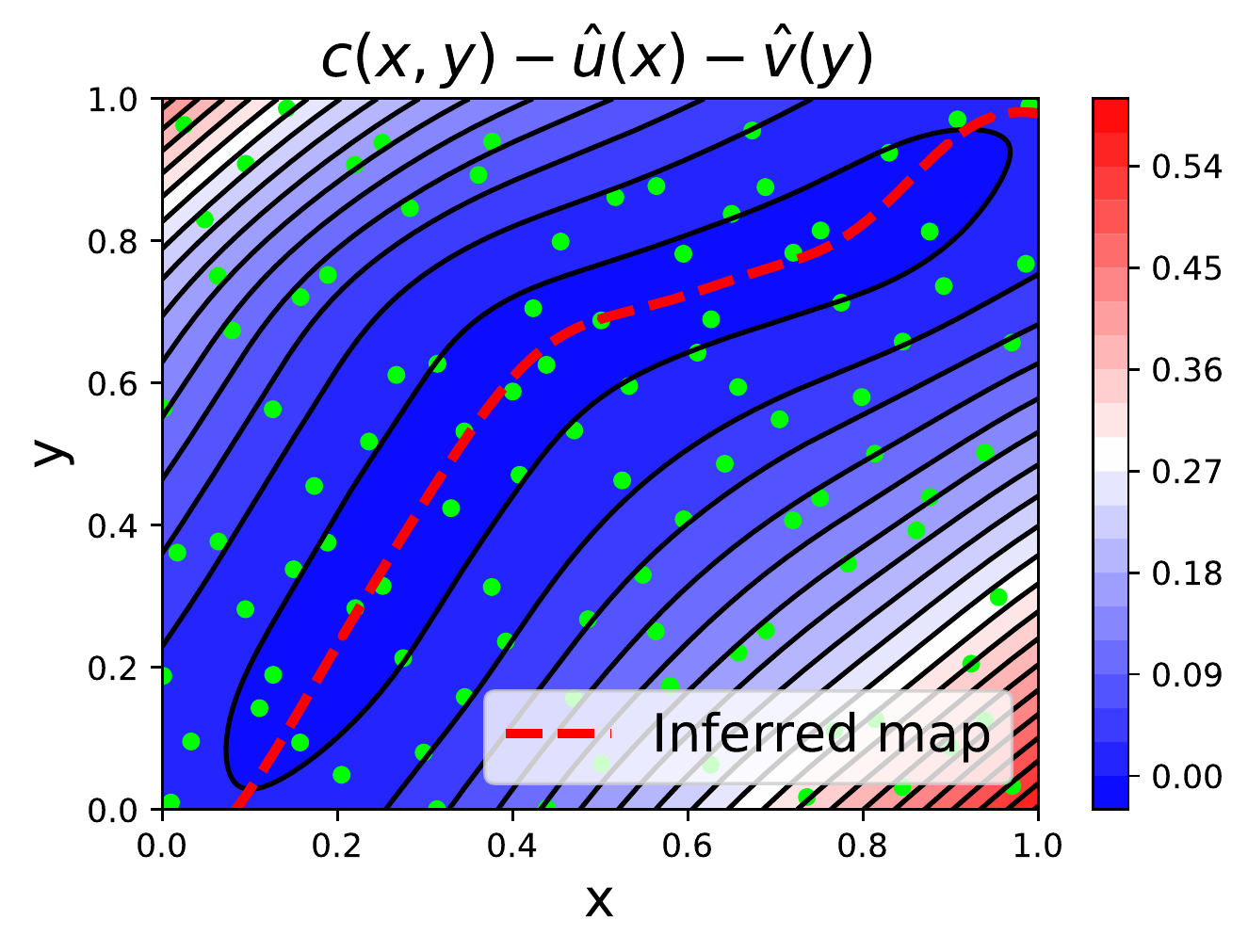}
    \includegraphics[width =  .32\textwidth]{figs/constraint_function_true.pdf}
    \caption{Effect of increasing the number of filling samples $\ell$ on the constraint model. \textit{(left)}: $\ell=50, {n_\mu} = {n_\nu} =25$, \textit{(middle)}: $\ell=100, {n_\mu} = {n_\nu} =25$, \textit{(right)}: true function.}
    \label{fig:constraint_fill}
\end{figure}

\begin{figure}[ht]
    \centering
    \includegraphics[width = .32\textwidth]{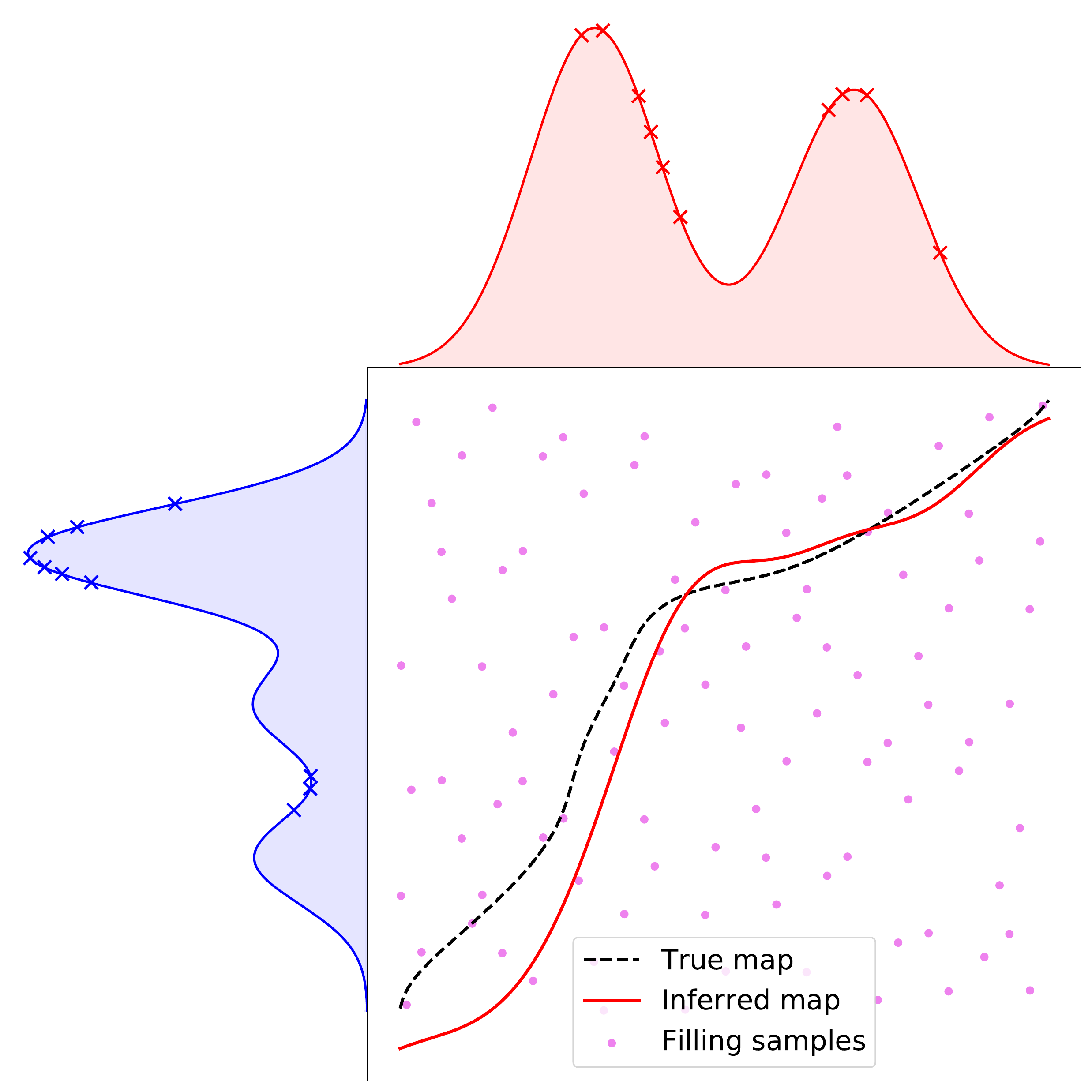}
    \includegraphics[width = .32\textwidth]{figs/transport_map_graph_nsamples_with_sobol_25_nfill_100.pdf}
    \includegraphics[width = .32\textwidth]{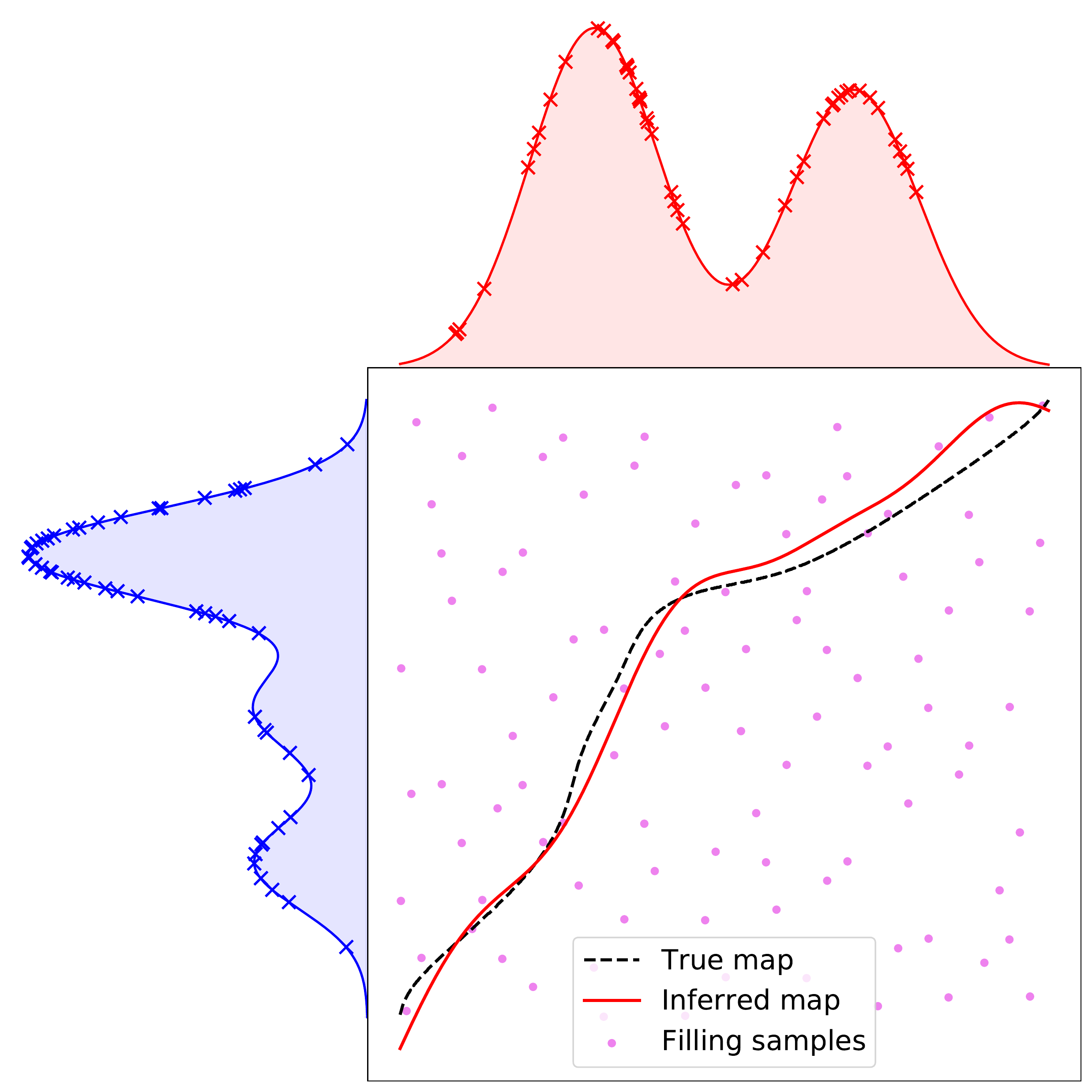}
    \caption{Effect of increasing the number of $\mu$ and $\nu$ samples on the transportation map. \textit{(left)}: $\ell=100, {n_\mu} = {n_\nu}=10$, \textit{(middle)}: $\ell=100, {n_\mu} =  {n_\nu} = 25$, \textit{(right)}: $\ell=100, {n_\mu} = {n_\nu} =50$.}
    \label{fig:transport_map_samples}
\end{figure}

\begin{figure}[ht]
    \centering
    \includegraphics[width = .32\textwidth]{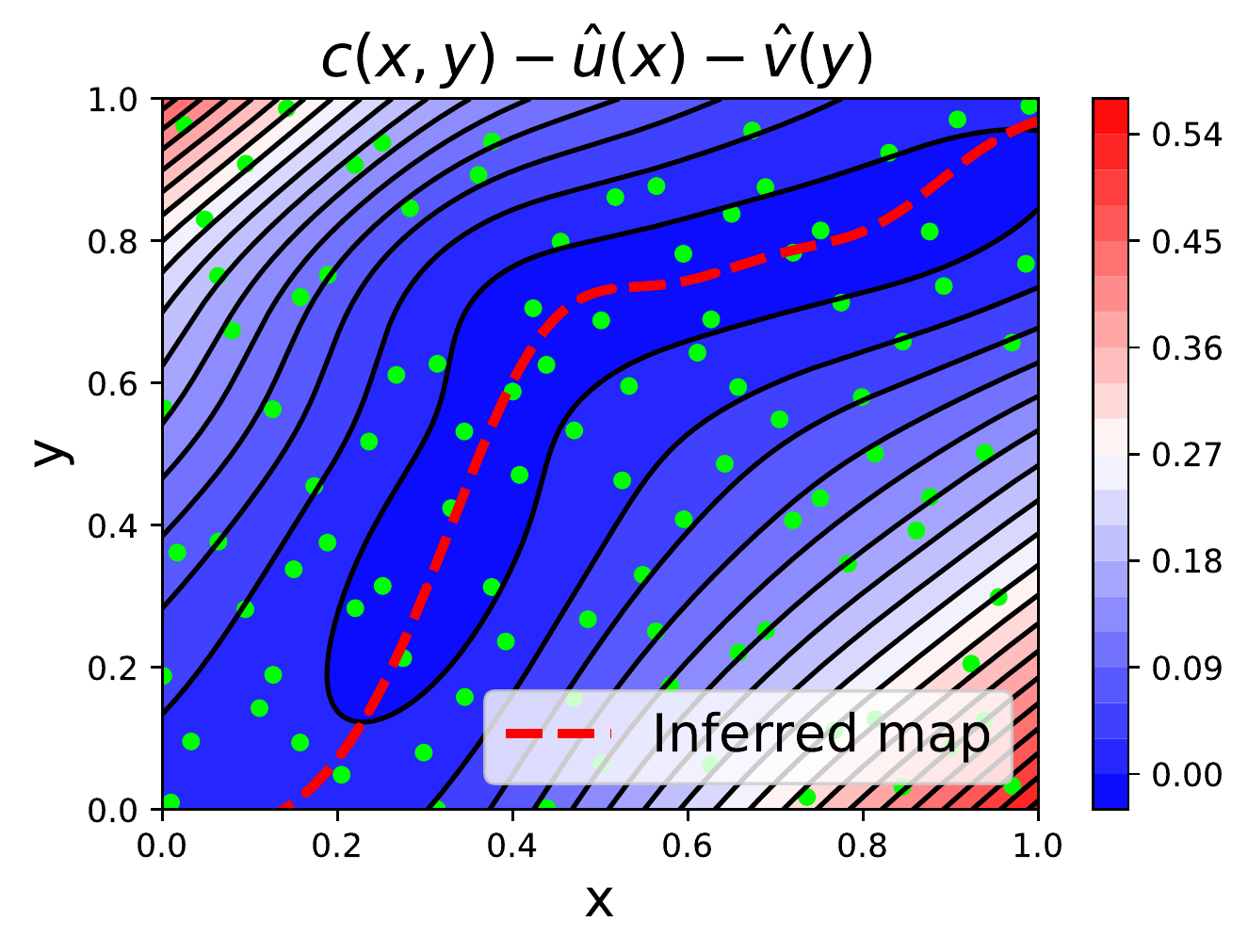}
    \includegraphics[width = .32\textwidth]{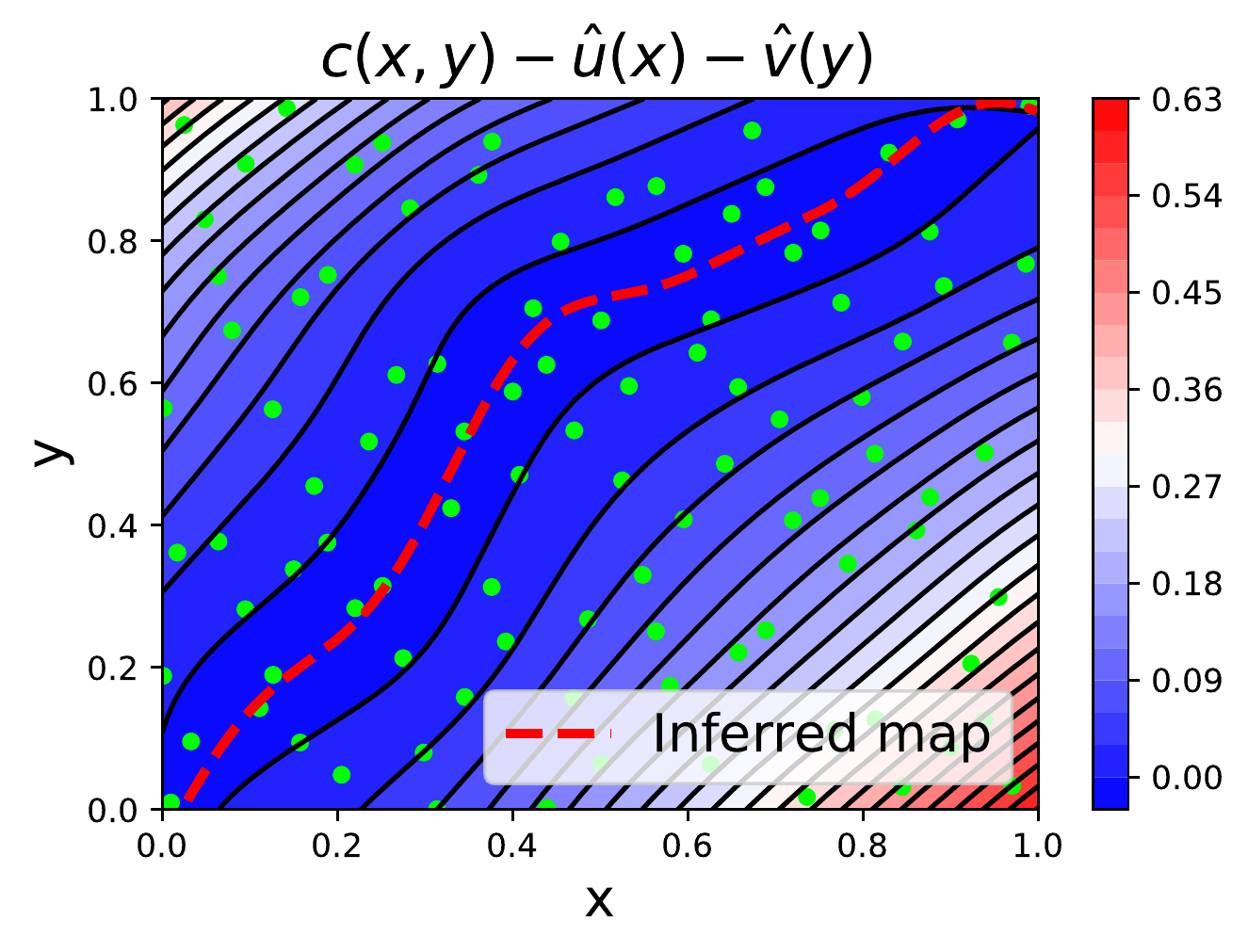}
    \includegraphics[width =  .32\textwidth]{figs/constraint_function_true.pdf}
    \caption{Effect of increasing the number of $\mu$ and $\nu$ samples on the constraint model. \textit{(left)}: $\ell=100, {n_\mu} = {n_\nu}=10$, \textit{(middle)}: $\ell=100, {n_\mu} = {n_\nu} =50$, \textit{(right)}: true function.}
    \label{fig:constraint_samples}
\end{figure}

\paragraph{1D transportation maps and dual constraint functions.}
We illustrate the algorithm described in \Cref{thm:dual_algorithm} in a 1D setting in \Cref{fig:transport_map_fill,fig:constraint_fill,fig:transport_map_samples,fig:constraint_samples}, by representing the inferred transportation map $\hat{T}$ obtained from $\hat{u}$, defined as $\hat{T} = x - \nabla_x \hat{u}(x)$, and the corresponding constraint function $\hat{h}(x, y) = \frac{1}{2}\|x-y\|^2 - \hat{u}(x) - \hat{v}(y)$. We sample $x_1, ..., x_{n_\mu}$ \textit{i.i.d.} from $\mu$ and $y_1, ..., y_{n_\nu}$ \textit{i.i.d.} from $\nu$, and use quasi-random samples $(\tx_1, \ty_1), ... (\tx_\ell, \ty_\ell)$ from a 2D Sobol sequence~\citep{sobol1967distribution}, and illustrate the effect of varying $n$, the number of $\mu$ and $\nu$ samples, and $\ell$, the number of space-filling samples. For $k_X, k_Y$ and $k_{XY}$, we use Gaussian kernels $k(x, y) = \exp(-\frac{\|x - y\|^2}{2\sigma^2})$ of fixed bandwidth $\sigma^2 = 0.1$, and scale the regularization parameters as $\lambda_1 = \frac{1}{\ell}$ and $\lambda_2 = \frac{1}{\sqrt{n}}$.


\paragraph{Convergence of $\hOT$ to $\OT$.}

In \Cref{fig:convergence_w_hat}, we compare $\hOT$ to the sampled optimal transport estimator on two 4D truncated Gaussian distributions $\mu$ and $\nu$ s.t.\ the optimal transportation map from one to another is linear. We progressively increase the number of $\mu$ and $\nu$ samples, averaging on $20$ random draws for each number of samples. The number of filling sample pairs $(\tx_i, \ty_i)$ is $\ell = 100 + {n}$, where $n = {n_\mu} = {n_\nu}$ is the number samples from $\mu$ and $\nu$. We select the best estimator $\hOT$ using a grid search on $(\lambda_1, \lambda_2)$. As such, this simulation does not provide a method for selecting those parameters, but rather illustrates that a good pair of parameters exists.

\begin{figure}[ht]
    \centering
    \includegraphics[width=.75\textwidth]{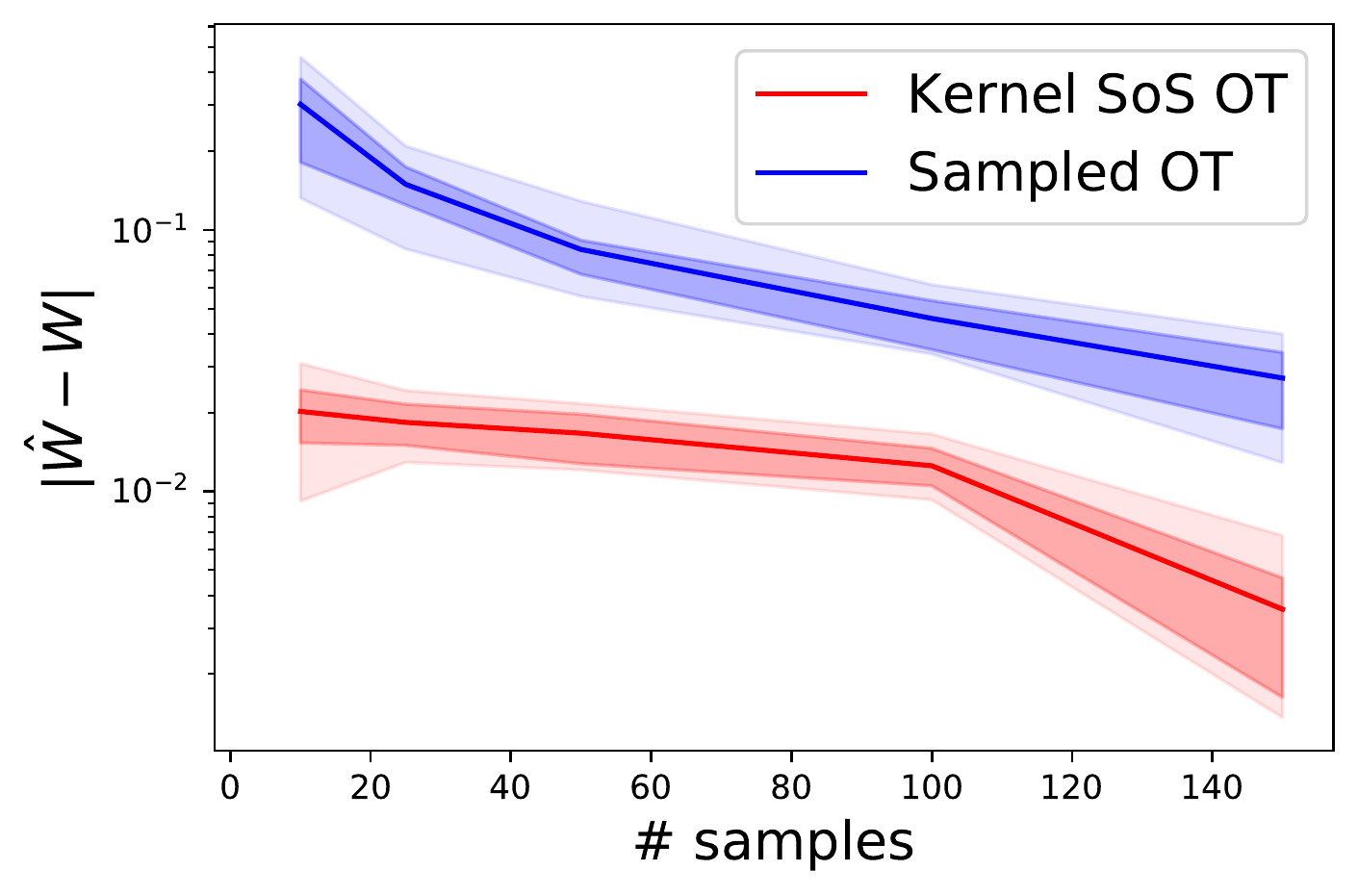}
    \caption{Convergence on 4D truncated Gaussian data with increasing number of samples ({\em left}). Full lines correspond to the average mean absolute error (MAE), shaded areas to $25\%-75\%$ and $10\%-90\%$ MAE quantiles. The parameters $\lambda_1, \lambda_2$ are selected via a grid search.}
    \label{fig:convergence_w_hat}
\end{figure}

\end{document}